\documentclass[11pt]{article}
\usepackage{mathrsfs}

\usepackage{amssymb}
\usepackage{amsmath}
\usepackage{amsfonts}
\usepackage{graphics}
\usepackage{epsfig}
\usepackage{enumerate}

\usepackage{color}

\allowdisplaybreaks

\newtheorem{theorem}{Theorem}[section]
\newtheorem{lemma}[theorem]{Lemma}
\newtheorem{definition}[theorem]{Definition}
\newtheorem{proposition}[theorem]{Proposition}

\newtheorem{remark}[theorem]{Remark}

\newenvironment{proof}{{\bf Proof:}}{~\hfill $\Box$}
\newenvironment{keywords}{{\bf Keywords: }}{}

\numberwithin{equation}{section}

\begin{document}

\author{Qian Lin \footnote{ {\it Email address}:
linqian1824@163.com}
\\
{\small Institute of Mathematical Economics, Bielefeld University,
Postfach   100131, 33501  Bielefeld, Germany  }}

\title{ Differentiability   of stochastic differential equations driven by the $G$-Brownian motion
}
\date{}

\maketitle \noindent
\begin{abstract} In this paper, we study the differentiability of the solutions of stochastic
differential equations driven by the $G$-Brownian motion with
respect to the initial data and the parameter.
\end{abstract}

\noindent
\begin{keywords}
$G$-expectation, differentiability, $G$-Brownian motion, stochastic
differential equations
\end{keywords}

\section{Introduction }

Motivated by  statistic model uncertainty, risk  measures,
superhedging in finance and $g$-expectation, Peng (see
\cite{Peng:2006}, \cite{Peng:2007}, \cite{Peng:2008}, \cite
{Peng:2009} and \cite{Peng:2010b})  introduced a  type of nonlinear
expectation--$G$-expectation. Together with the notion of the
$G$-expectation,  Peng  introduced the related $G$-normal
distribution and the $G$-Brownian motion and established an It\^{o}
calculus for the $G$-Brownian motion.

In the framework of $G$-expectations,  many interesting facts that
are rather different from classical cases have been studied. The
$G$-Brownian motion has a very rich and interesting new structure
which non-trivially generalizes the classical one. A very
interesting new phenomenon of the $G$-Brownian motion $B$ is that
its quadratic process $\langle B\rangle$ is a stochastic process and
has independent and stationary increments which are identically
distributed.

 Peng  (see \cite{Peng:2006}, \cite{Peng:2007},
\cite{Peng:2008}, \cite{Peng:2010b})   derived the existence and
uniqueness of the solution to stochastic differential equations with
Lipschitz coefficients on the coefficients driven by the
$G$-Brownian motion. Recently, Gao \cite{Gao} studied  pathwise
properties and homeomorphic flows for stochastic differential
equations driven by the $G$-Brownian motion. Lin \cite{L} obtained
the continuity of stochastic integrals with respect to the
$G$-Brownian motion and solutions of stochastic differential
equations depending on parameters driven by the $G$-Brownian motion.

The purpose of  this paper is to investigate the differentiability
of solutions of stochastic differential equations driven by the
$G$-Brownian motion with respect to the initial data and the
parameter, which  has some potential applications, e.g., to go
further to obtain the maximum principle for stochastic optimal
control systems driven by a $G$-Brownian motion.

 The rest of the paper is organized as follows. In Section 2, we
  introduce some notations and preliminaries,  which we will need in what follows.
  In Section 3, we study the differentiability of the solutions of SDEs  driven by the
   $G$-Brownian motion with respect to the initial data. In Section 4, we study the differentiability of
the solutions of SDEs  driven by the $G$-Brownian motion with
respect to the parameter.

\section{Notations and preliminaries  }

The objective of  this section is to briefly recall  the theory of
the $G$-expectation and the related $G$-stochastic analysis (see
Peng \cite{Peng:2006}, \cite{Peng:2007}, \cite{Peng:2008} and \cite
{Peng:2009} for more details).

Let $\Omega$ be a given nonempty set and $\mathcal {H}$ be a linear
space of real functions defined on $\Omega$ such that if
$x_{1},\cdot\cdot\cdot,x_{m}\in \mathcal {H}$ then
$\varphi(x_{1},\cdot\cdot\cdot,x_{m})\in \mathcal {H}$, for each
$\varphi\in C_{l,lip}(\mathbb{R}^{m})$, where
$C_{l,lip}(\mathbb{R}^{m})$ denotes the linear space of functions
$\varphi$ satisfying
$$|\varphi(x)-\varphi(y)|\leq C(1+|x|^{n}+|y|^{n})|x-y|,\ \text{for all}\ x, y\in \mathbb{R}^{m},$$
for some $C>0$ and $\ n\in\mathbb{N}$, both depending on $\varphi$.
The space $\mathcal {H}$ is considered as a set of random
variables.\vskip2mm

\begin{definition}
{\bf A Sublinear expectation} $\mathbb{\hat{E}}$ on $\mathcal {H}$
is a functional $\mathbb{\hat{E}}:\mathcal {H}\mapsto \mathbb{R}$,
which satisfies the following properties: for all $X, Y \in \mathcal
{H}$, we have
\begin{enumerate}
\item[(i)] {\bf Monotonicity:} If $X\geq Y $, then
$\mathbb{\hat{E}}[X]\geq\mathbb{\hat{E}}[Y]$.
\item[(ii)] {\bf Constant preserving:}
$\mathbb{\hat{E}}[c]=c$, for all $c\in \mathbb{R}$.
\item[(iii)] {\bf Self-dominated property:}
$\mathbb{\hat{E}}[X]-\mathbb{\hat{E}}[Y]\leq\mathbb{\hat{E}}[X-Y]$.
\item[(iv)] {\bf Positive homogeneity:}
$\mathbb{\hat{E}}[\lambda X]=\lambda\mathbb{\hat{E}}[X],$ for all $
\lambda \geq 0$.
\end{enumerate}
The triple $(\Omega, \mathcal {H}, \mathbb{\hat{E}})$ is called a
sublinear expectation space.
\end{definition}\vskip2mm

\begin{definition} In a sublinear expectation space $(\Omega, \mathcal {H},
\mathbb{\hat{E}})$, a random vector $Y=(Y_{1},\cdots, Y_{n}),
Y_{i}\in \mathcal {H}$, is said to be independent under $
\mathbb{\hat{E}}$ from  another random vector $X=(X_{1},\cdots,
X_{m}), X_{i}\in \mathcal {H}$,  if for each test function $\varphi
\in C_{l,lip}(\mathbb{R}^{m+n})$, we have
$$\mathbb{\hat{E}}[\varphi(X, Y)]=\mathbb{\hat{E}}[\mathbb{\hat{E}}[\varphi(x, Y)]_{x=X}].$$
\end{definition}

\begin{definition}{\bf($G$-normal distribution)}
 Let be given two reals $\underline{\sigma}, \overline{\sigma}$ with
$0\leq\underline{\sigma}\leq \overline{\sigma}.$ A random variable
$\xi$ in a sublinear expectation space $(\Omega, \mathcal {H},
\mathbb{\hat{E}})$ is called $G_{\underline{\sigma},
\overline{\sigma}}$-normal distributed, denoted by $\xi\sim \mathcal
{N}(0, [{\underline{\sigma}^{2}, \overline{\sigma}}^{2}])$,  if for
each $\varphi \in C_{l,lip}(\mathbb{R})$, the following function
defined by
$$u(t, x):=\mathbb{\hat{E}}[\varphi(x+\sqrt{t}\xi)], \quad (t,x)\in [0, \infty)\times\mathbb{R},$$
is the unique continuous  viscosity solution with the polynomial
growth of the following parabolic partial differential equation :
$$\left\{\begin{array}{l l}
      \partial_{t}u(t, x)=G(\partial_{xx}^{2}u(t, x)),\quad (t, x)\in[0,
\infty)\times\mathbb{R},\\
       u(0, x)=\varphi(x).
         \end{array}
  \right.$$
 Here  $G=G_{\underline{\sigma},
\overline{\sigma}}$
 is the following sublinear function parameterized by $\underline{\sigma}$
 and $\overline{\sigma}$:
 $$G(\alpha)=\frac{1}{2}(\overline{\sigma}^{2}\alpha^{+}-\underline{\sigma}^{2}\alpha^{-}), \quad \alpha \in \mathbb{R}$$
 (Recall that $\alpha^{+}=\text{max}\{0,\alpha\}$ and
$\alpha^{-}=-\text{min}\{0,\alpha\}$). For simplicity, we suppose
that $\overline{\sigma}=1$ and $\underline{\sigma}=\sigma, 0\leq
\sigma \leq 1, $ in what follows.
\end{definition}

 Throughout this paper, we let $\Omega=C_{0}(\mathbb{R^{+}})$ be the space of all
real valued continuous functions $(\omega_{t})_{t\in
\mathbb{R^{+}}}$ with $\omega_{0}=0$, equipped with the distance
$$\rho(\omega^{1}, \omega^{2})
=\sum\limits_{i=1}^{\infty}2^{-i}\Big[(\max\limits_{t\in[0,i]}|\omega_{t}^{1}-\omega_{t}^{2}|)\wedge1\Big],
\ \omega_{t}^{1}, \omega_{t}^{2}\in\Omega.$$

For each $T > 0$, we consider the following space of random
variables:
\begin{eqnarray*}
L_{ip}^{0}(\mathcal
{F}_{T}):=\Big\{X(\omega)=\varphi(\omega_{{t_{1}}} \cdots,
\omega_{{t_{m}}}) \ | \ t_{1}, \cdots, t_{m}\in [0,T],  \text{ for
all} \ \varphi\in C_{l,lip}(\mathbb{R}^{m}), \ m\geq1\Big\}.
\end{eqnarray*}
 We notice
that $X,Y \in L_{ip}^{0}(\mathcal {F}_{t})$ implies $X\cdot Y \in
L_{ip}^{0}(\mathcal {F}_{t})$ and $|X|\in L_{ip}^{0}(\mathcal
{F}_{t})$. We further define
$$L_{ip}^{0}(\mathcal {F})=\bigcup_{n=1}^{\infty}L_{ip}^{0}(\mathcal
{F}_{n}).$$

 We will work on the canonical space $\Omega$ and set
$B_{t}(\omega)=\omega_{t}$, $t\in [0,\infty)$, for $\omega\in
\Omega.$ We now recall  a sublinear expectation $\mathbb{\hat{E}}$
defined
 on $\mathcal {H}_{T}^{0}=L_{ip}^{0}(\mathcal
{F}_{T})$ as well as on $\mathcal {H}^{0}=L_{ip}^{0}(\mathcal {F})$.
For this, we consider the function $G(a)= \frac{1 }{2}
(a^{+}-\sigma^{2}a^{-}), a \in \mathbb{R}$ (with
$\overline{\sigma}^{2}=1, \underline{\sigma}^{2}=\sigma^{2}\in [0,
1]$), and we apply the following procedure: for each $X\in \mathcal
{H}^{0}$ with
$$X=\varphi(B_{t_{1}}-B_{t_{0}}, B_{t_{2}}-B_{t_{1}}, \cdots, B_{t_{m}}-B_{t_{m-1}})$$
for some $m\geq 1, \varphi\in C_{l,lip}(\mathbb{R}^{m})$ and
$0=t_{0}\leq t_{1}\leq\cdots \leq t_{m}<\infty$, we set
\begin{eqnarray*}
&&\mathbb{\hat{E}}[\varphi(B_{t_{1}}-B_{t_{0}},
B_{t_{2}}-B_{t_{1}}, \cdots, B_{t_{m}}-B_{t_{m-1}})]\\
&&=\mathbb{\widetilde{E}}[\varphi(\sqrt{t_{1}-t_{0}}\xi_{1},
\sqrt{t_{2}-t_{1}}\xi_{2}, \cdots, \sqrt{t_{m}-t_{m-1}}\xi_{m})],
\end{eqnarray*}
where $(\xi_{1},\xi_{2},\cdots, \xi_{m})$ is an m-dimensional
G-normal distributed random vector in a sublinear expectation space
$(\widetilde{\Omega}, \widetilde{\mathcal {H}},
\widetilde{\mathbb{E}})$ such that $\xi_{i}\sim\mathcal {N}(0,
[\sigma^{2},1])$ and  $\xi_{i+1}$ is independent from $(\xi_{1},
\cdots, \xi_{i})$ for each $i=1, 2, \cdots, m-1$.

The related conditional expectation of
$X=\varphi(B_{t_{1}}-B_{t_{0}}, B_{t_{2}}-B_{t_{1}}, \cdots,
B_{t_{m}}-B_{t_{m-1}})$ under $\mathcal {H}_{t_{j}}^{0}$ is defined
by
\begin{eqnarray*}
\mathbb{\hat{E}}[X|\mathcal
{H}_{t_{j}}^{0}]&=&\mathbb{\hat{E}}[\varphi(B_{t_{1}}-B_{t_{0}},
B_{t_{2}}-B_{t_{1}}, \cdots, B_{t_{m}}-B_{t_{m-1}})|\mathcal
{H}_{t_{j}}^{0}]\\
&=&\psi(B_{t_{1}}-B_{t_{0}}, B_{t_{2}}-B_{t_{1}}, \cdots,
B_{t_{j}}-B_{t_{j-1}}),
\end{eqnarray*}
where
$$\psi(x_{1}, x_{2}, \cdots,
x_{j})=\mathbb{\widetilde{E}}[\varphi(x_{1}, x_{2}, \cdots, x_{j},
\sqrt{t_{j+1}-t_{j}}\xi_{j+1}, \cdots,
\sqrt{t_{m}-t_{m-1}}\xi_{m})],$$ $(x_{1}, x_{2}, \cdots,
x_{j})\in\mathbb{R}^{j}, 0\leq j\leq m.$

\vspace{2mm} For $p\geq 1$,
$\|X\|_{p}=\mathbb{\hat{E}}^{\frac{1}{p}}[|X|^{p}]$,  $X \in
L_{ip}^{0}(\mathcal {F})$, defines a norm on  $L_{ip}^{0}(\mathcal
{F})$. Let $\mathcal {H}=L_{G}^{p}(\mathcal {F})$ (resp. $\mathcal
{H}_{t}=L_{G}^{p}(\mathcal {F}_{t})$) be the completion of
$L_{ip}^{0}(\mathcal {F})$ (resp. $L_{ip}^{0}(\mathcal {F}_{t})$)
under the norm $\|\cdot\|_{p}$. Then the space $(L_{G}^{p}(\mathcal
{F}), \|\cdot\|_{p})$ is a Banach space and the operators
$\mathbb{\hat{E}}[\cdot] $ (resp. $\mathbb{\hat{E}}[\cdot|\mathcal
{H}_{t}]$) can be continuously extended to the Banach space
$L_{G}^{p}(\mathcal {F})$ (resp. $L_{G}^{p}(\mathcal {F}_{t})$).
Moreover, we have $L_{G}^{p}(\mathcal {F}_{t})\subseteq
L_{G}^{p}(\mathcal {F}_{T})\subset L_{G}^{p}(\mathcal {F})$, for all
$0\leq t \leq T <\infty$.\vskip2mm

\begin{definition}
  The expectation $\mathbb{\hat{E}}: L_{G}^{p}(\mathcal {F})\mapsto \mathbb{R}$  defined through the above procedure is
called G-expectation.
\end{definition}

\begin{proposition}\label{p2}
  For all
$t, s\in [0, \infty)$, we list the properties of
$\mathbb{\hat{E}}[\cdot|\mathcal {H}_{t}]$ that hold for all $X, Y
\in L_{G}^{p}(\mathcal {F}):$

\begin{enumerate}[(i)]
\item  If $X\geq Y$, then $\mathbb{\hat{E}}[X|\mathcal
{H}_{t}]\geq \mathbb{\hat{E}}[Y|\mathcal {H}_{t}]$;
\item $\mathbb{\hat{E}}[\eta|\mathcal
{H}_{t}]=\eta$, for all $\eta\in L_{G}^{p}(\mathcal {F}_{t})$;
\item $\mathbb{\hat{E}}[X|\mathcal
{H}_{t}]- \mathbb{\hat{E}}[Y|\mathcal
{H}_{t}]\leq\mathbb{\hat{E}}[X-Y|\mathcal {H}_{t}];$
\item If $\mathbb{\hat{E}}[Y|\mathcal
{H}_{t}]=-\mathbb{\hat{E}}[-Y|\mathcal {H}_{t}],$  then
$\mathbb{\hat{E}}[X+Y|\mathcal {H}_{t}]= \mathbb{\hat{E}}[X|\mathcal
{H}_{t}]+\mathbb{\hat{E}}[Y|\mathcal {H}_{t}];$
\item $\mathbb{\hat{E}}[\mathbb{\hat{E}}[X|\mathcal
{H}_{t}]|\mathcal {H}_{s}]=\mathbb{\hat{E}}[X|\mathcal {H}_{t\wedge
s}]$, and, in particular,
$\mathbb{\hat{E}}[\mathbb{\hat{E}}[X|\mathcal
{H}_{t}]]=\mathbb{\hat{E}}[X]$.
\end{enumerate}

\end{proposition}

\vspace{2mm}
 For $p\geq 1$ and arbitrary but fixed $0<T<\infty$, we first consider the
following space of step processes:
\begin{eqnarray*}
M_{G}^{p, 0}(0,T) &=& \bigg\{\eta:
\eta_{t}=\sum\limits_{j=0}^{N-1}\xi_{j}I_{[t_{j},t_{j+1})},
0=t_{0}<t_{1}<\cdots<t_{N}=T, \\&& \hskip 1cm\xi_{j}\in
L_{G}^{p}(\mathcal {F}_{t_{j}}), j=0,\cdots, N-1, \text{for all}\
N\geq 1\bigg\},
\end{eqnarray*}
and we  define the following norm in $M_{G}^{p, 0}(0,T)$:

$$\parallel \eta \parallel_{p}=\bigg(\frac{1}{T}\int_{0}^{T}\mathbb{\hat{E}}\Big[|\eta_{t}|^{p}\Big]dt\bigg)^{\frac{1}{p}}
=\bigg(\frac{1}{T}\sum\limits_{j=0}^{N-1}\mathbb{\hat{E}}\Big[|\xi_{j}|^{p}\Big](t_{j+1}-t_{j})\bigg)^{\frac{1}{p}}$$

Finally, we denote by $M_{G}^{p}(0,T)$ the completion of $M_{G}^{p,
0}(0,T)$ under the norm $\parallel \cdot\parallel_{p}.$

\begin{definition}
 A process $B=\{B_{t},t\geq 0\}$ in a sublinear expectation space $(\Omega, \mathcal {H},
 \mathbb{\hat{E}})$ is called a G-Brownian motion if for each
 $n\in\mathbb{N}$ and $0\leq t_{1}\leq\cdots \leq t_{n}<\infty$,
 $B_{t_{1}}, \cdots, B_{t_{n}}\in \mathcal {H}$ and the following
 property satisfied:
\begin{enumerate}
\item[(i)] $B_{0}=0$;
\item[(ii)] For each $t, s \geq0$, the difference $B_{t+s}-B_{t}$ is
$\mathcal {N}(0, \ [\sigma^{2} s, s])$-distributed and is
independent from $(B_{t_{1}}, \cdots, B_{t_{n}})$, for all
$n\in\mathbb{N}$ and $0\leq t_{1}\leq\cdots \leq t_{n}\leq t$.
\end{enumerate}
\end{definition}

\begin{remark}
The  canonical process $(B_{t})_{t\geq 0}$ in $(\Omega, \mathcal
{H})$, $\Omega=C_{0}(\mathbb{R}_{+})$, endowed with the
$G$-expectation $\mathbb{\hat{E}}$ is a $G$-Brownian motion.

\end{remark}

\begin{remark}
  In \cite{Peng:2006}, \cite{Peng:2007}, \cite{Peng:2008} and \cite{Peng:2010b},
   Peng established a stochastic calculus of
It\^o's type with respect to the G-Brownian motion  and its
quadratic variation process.  Peng derived an It\^o's formula and
moreover, he obtained the existence and uniqueness of the solution
to stochastic differential equations with Lipschitz coeffcients
driven by the $G$-Brownian motion.
\end{remark}

In order to simplify notations, we will only consider one
dimensional stochastic differential equations driven by the
$G$-Brownian motion. Multidimensional stochastic differential
equations driven by the $G$-Brownian motion can be studied in a
similar way.

In what follows, we let  $T$ be an arbitrarily fixed time horizon
and $\{B_t,\ 0\leq t\leq T\}$  a one dimensional G-Brownian motion,
and we consider the following stochastic differential equation:
\begin{eqnarray}\label{eq}
X_t^{x}=x+\int_0^t b(s, X_s^{x})ds+\int_0^t \sigma(s,
X_s^{x})dB_{s}+\int_0^t h(s, X_s^{x})d\langle B \rangle_{s}, \ t\in
[0,T],
\end{eqnarray}
where the initial condition $x\in \mathbb{R}$ is  given and
$b(\cdot, \cdot), \sigma(\cdot, \cdot), h(\cdot, \cdot)
:[0,T]\times\mathbb{R}\rightarrow\mathbb{R}$. Let us make the
following assumption:

\begin{list}{}{\setlength{\itemindent}{0cm}}

\item[(H)] There exists a positive constant $L$ such that, for all
 $x ,x^{\prime}\in\mathbb{R}, t\in[0,T]$,
\begin{eqnarray*}
&&|b(t, x)-b(t, x^{\prime})|+|\sigma(t, x)-\sigma(t,
x^{\prime})|+|h(t, x)-h(t, x^{\prime})|\leq L|x-x^{\prime}|.
\end{eqnarray*}
\end{list}

  \begin{lemma} \label{weiyi}
Under  the assumption (H), there exists a unique solution $X\in
M_{G}^{2}(0,T)$ of the stochastic differential equation (\ref{eq}).
\end{lemma}

\begin{lemma}\label{guji}
 For every $p\geq 2 $, there exists a positive constant $C_{p}$ such
 that, for any $T>0$ and $\eta\in M_{G}^{p}(0,T)$,
\begin{eqnarray*}\label{inequ1}
\mathbb{\hat{E}}\Big[\sup\limits_{t\in[0,T]}|\int_0^t
\eta_{s}ds|^{p}\Big] &\leq &T^{p-1}\int_0^T
\mathbb{\hat{E}}\Big[|\eta_{s}|^{p}\Big]ds,\\
\mathbb{\hat{E}}\Big[\sup\limits_{t\in[0,T]}|\int_0^t
\eta_{s}d\langle B\rangle_{s}|^{p}\Big]&\leq & T^{p-1}\int_0^T
\mathbb{\hat{E}}\Big[|\eta_{s}|^{p}\Big]ds,\\
\mathbb{\hat{E}}\Big[\sup\limits_{t\in[0,T]}|\int_0^t
\eta_{s}dB_{s}|^{p}\Big]&\leq & C_{p}T^{\frac{p}{2}-1}\int_0^T
\mathbb{\hat{E}}\Big[|\eta_{s}|^{p}\Big]ds.
\end{eqnarray*}
\end{lemma}

\section{Differentiability of solutions of SDEs  driven by the $G$-Brownian motion with respect to the initial data}

The objective of this section is to study the differentiability of
the solutions of stochastic differential equations  driven by the
$G$-Brownian motion with respect to the initial data.  We first give
a definition of continuity and differentiability as follows:
\begin{definition}
A process $\{Y_{t}\}_{t\in[0,T]}$ is said to be continuous in
$L_{G}^{2}$ if
$$\lim\limits_{h\rightarrow0}\mathbb{\hat{E}}[|Y_{t}-Y_{t+h}|^{2}]=0.$$
\end{definition}

\begin{definition}
Let  $g(x)=g(x_{1},\cdots, x_{n})$ and  $f(x)=f(x_{1},\cdots,
x_{n})$ be random functions. For some $1\leq i\leq n$, if
\begin{eqnarray*}
\lim\limits_{h\rightarrow0}\mathbb{\hat{E}}[|\frac{g(x_{1},\cdots,x_{i-1},x_{i},x_{i+1},\cdots,
x_{n})-g(x_{1},\cdots, x_{n})}{h}-f(x_{1},\cdots, x_{n})|^{2}]=0.
\end{eqnarray*}
then we define the partial derivative of  $g(x_{1},\cdots, x_{n})$
with respect to $x_{i}$ as  $f(x_{1},\cdots, x_{n})$.  We write
\begin{eqnarray*}
\frac{\partial g(x_{1},\cdots, x_{n})}{\partial
x_{i}}=g_{x_{i}}(x_{1},\cdots, x_{n})=f(x_{1},\cdots, x_{n}).
\end{eqnarray*}
\end{definition}

We also need the following propositions.
\begin{proposition}\label{pr1}
Under  the assumption (H), for all $p\geq0$, we have  the following
estimate for the solution $X$ of SDE (\ref{eq}): for all $r\in
[0,T]$,
$$\mathbb{\hat{E}}[\sup\limits_{t\in[0,r]}|X_{t}^{x}|^{p}]\leq C<+\infty.$$
Here the constant $C=C(p,x,r,L)$.
\end{proposition}

\begin{proof}
We give the proof in three steps.

Step 1.  For $p\geq 2$, by equation(\ref{eq}) we have
\begin{eqnarray*}
|X_{t}^{x}|^{p}\leq 4^{p-1}\Big(|x|^{p}+|\int_0^t b(s,
X_s^{x})ds|^{p}+ |\int_0^t \sigma(s, X_s^{x})dB_{s}|^{p} +|\int_0^t
h(s, X_s^{x})d\langle B\rangle_{s}|^{p}\Big).
\end{eqnarray*}
From the subadditivity of the $G$-expectation  it follows that
\begin{eqnarray*}
\mathbb{\hat{E}}[\sup\limits_{t\in[0,r]}|X_{t}^{x}|^{p}]
 &\leq & 4^{p-1}\Big(|x|^{p}+\mathbb{\hat{E}}\Big[\sup\limits_{t\in[0,r]}|\int_0^t b(s,
X_s^{x})ds|^{p}\Big]\\
&&+\mathbb{\hat{E}}\Big[\sup\limits_{t\in[0,r]}|\int_0^t \sigma(s,
X_s^{x})dB_{s}|^{p}\Big]
+\mathbb{\hat{E}}\Big[\sup\limits_{t\in[0,r]}|\int_0^t h(s,
X_s^{x})d\langle B\rangle_{s}|^{p}\Big]\Big).
\end{eqnarray*}
Thus, from (H), Lemma \ref{guji}  it follows that
\begin{eqnarray*}
\mathbb{\hat{E}}[\sup\limits_{t\in[0,r]}|X_{t}^{x}|^{p}]
&\leq &C \Big[|x|^{p}+r^{p-1}\int_0^r \Big(\mathbb{\hat{E}}[|b(s,0)|^{p}]+\mathbb{\hat{E}}[|X_s^{x}|^{p}]\Big)ds\\
& & +r^{\frac{p}{2}-1}\int_0^r
\Big(\mathbb{\hat{E}}[|\sigma(s,0)|^{p}]+\mathbb{\hat{E}}[|X_s^{x}|^{p}]\Big)ds
\\ && +r^{p-1}\int_0^T \Big(\mathbb{\hat{E}}[|h(s,0)|^{p}]+\mathbb{\hat{E}}[|X_s^{x}|^{p}]\Big)ds\Big]\\
&\leq &C \Big(1+\int_0^r
\mathbb{\hat{E}}[|X_s^{x}|^{p}]ds\Big)\\
&\leq &C \Big(1+\int_0^r
\mathbb{\hat{E}}[\sup\limits_{s'\in[0,s]}|X_{s'}^{x}|^{p}]ds\Big).
\end{eqnarray*}
Then Gronwall's inequality yields
$$\mathbb{\hat{E}}[\sup\limits_{t\in[0,r]}|X_{t}^{x}|^{p}]\leq C<+\infty.$$

Step 2. For all $1\leq p < 2$, from H\"{o}lder inequality  under the
$G$-expectation and Step 1 it follows that
$$\mathbb{\hat{E}}[\sup\limits_{t\in[0,r]}|X_{t}^{x}|^{p}]\leq
\mathbb{\hat{E}}[\sup\limits_{t\in[0,r]}|X_{t}^{x}|^{2p}]^{\frac{1}{2}}\leq
C<+\infty.$$

Step 3. For all $0< p < 1$, since
$$|X_{t}^{x}|^{p}\leq |X_{t}^{x}|^{p}1_{\{|X_{t}^{x}|^{p}\leq1\}}+|X_{t}^{x}|^{p}1_{\{|X_{t}^{x}|^{p}\geq1\}}
\leq 1+|X_{t}^{x}|^{1+p},$$ then from Step 2, we have
$$\mathbb{\hat{E}}[\sup\limits_{t\in[0,r]}|X_{t}^{x}|^{p}]\leq
\mathbb{\hat{E}}[1+\sup\limits_{t\in[0,r]}|X_{t}^{x}|^{1+p}]\leq
C<+\infty.$$
 The proof is complete.
\end{proof}

Similar to the proof of Proposition \ref{pr1}, we have the following
proposition.\vskip2mm
\begin{proposition}\label{pr2}
 Let us assume (H). Then for every $p\geq2$, there exists a positive constant $C$ such
that
$$\mathbb{\hat{E}}[|X_t^{x}-X_{t}^{y}|^{p}]
\leq C|x-y|^{p}, \ \text{for all} \ t \in [0,T],$$ where $C$ depends
only on $p,T$.
\end{proposition}

We have the following differentiability result with respect to the
initial data.\vskip2mm
\begin{theorem}\label{th1}
For all $t\in [0,T]$, if  $b_{x}(t, \cdot),\sigma_{x}(t, \cdot),
h_{x}(t, \cdot)\in C_{l,lip}(\mathbb{R})$ and are bounded, then
$X^{x}_{t}$ is  differentiable in $L_{G}^{2}$ with respect to $x$.
Moreover, $Y_{t}^{x}:=\frac{\partial X^{x}_{t}}{\partial x}$
satisfies the following stochastic differential equation
\begin{eqnarray}\label{eq6}
Y_t^{x}=1+\int_0^t b_{x}(s, X_s^{x})Y_s^{x}ds+\int_0^t \sigma_{x}(s,
X_s^{x})Y_s^{x}dB_{s}+\int_0^t h_{x}(s, X_s^{x})Y_s^{x}d\langle
B\rangle_{s}, \ t\in [0,T].
\end{eqnarray}
\end{theorem}

\begin{proof}
Let $h\neq 0$ be small. For simplicity, we put
\begin{eqnarray*}
X_t:=X_t^{x},\ Y_t:=Y_t^{x}, \ \widetilde{X}_t:=X_t^{x+h}, \
Z^{h}_t:=\frac{\widetilde{X}_t-X_t}{h}.
\end{eqnarray*}
Then we have
 \begin{eqnarray}\label{eq8}
Z_t^{h}&=&1+\frac{1}{h}\int_0^t [b(s, \widetilde{X}_s)-b(s,
X_s)]ds+\frac{1}{h}\int_0^t [\sigma(s, \widetilde{X}_s)-\sigma(s, X_s)]dB_{s}\nonumber\\
&&+\frac{1}{h}\int_0^t [h(s, \widetilde{X}_s)-h(s, X_s)]d\langle B
\rangle_{s}.
\end{eqnarray}
Since   $b_{x}(t, \cdot),\sigma_{x}(t, \cdot), h_{x}(t, \cdot)\in
C_{l,lip}(\mathbb{R})$,  $t\in [0,T]$, we have
 \begin{eqnarray}\label{eq5}
Z_t^{h}&=&1+\int_0^t \int_0^1 b_{x}(s,
X_s+\theta(\widetilde{X}_s-X_s))d\theta Z_s^{h} ds
\nonumber\\&&+\int_0^t \int_0^1 \sigma_{x}(s,
X_s+\theta(\widetilde{X}_s-X_s))d\theta Z_s^{h}dB_{s}\nonumber\\
&&+\int_0^t \int_0^1 h_{x}(s,
X_s+\theta(\widetilde{X}_s-X_s))d\theta Z_s^{h}d\langle B
\rangle_{s}\nonumber\\
&\doteq&1+I_{1}+I_{2}+I_{3}.
\end{eqnarray}
Therefore, for some $k\in\mathbb{N}$,
 \begin{eqnarray*}\label{}
&&\mathbb{\hat{E}}[\sup\limits_{t\in [0,T]}|I_{1}-\int_0^t b_{x}(s,
X_s)Y_sds|^{2}]\\
&\leq &2\mathbb{\hat{E}}[(\int_0^T \int_0^1 |b_{x}(s,
X_s+\theta(\widetilde{X}_s-X_s))|d\theta |Z_s^{h}-Y_s|
ds)^{2}]\\&&+2\mathbb{\hat{E}}[(\int_0^T \int_0^1 |b_{x}(s,
X_s+\theta(\widetilde{X}_s-X_s))- b_{x}(s, X_s)|d\theta |Y_s|
ds)^{2}]\\
&\leq &C \mathbb{\hat{E}}[\int_0^T |Z_s^{h}-Y_s|^{2}
ds]\\&&+C\mathbb{\hat{E}}[(\int_0^T
(|\widetilde{X}_s-X_s|+|\widetilde{X}_s-X_s|^{k+1}+
|X_s|^{k}|\widetilde{X}_s-X_s|)
 |Y_s|ds)^{2}].
\end{eqnarray*}
For all $\varepsilon>0,$ from
$2ab\leq\frac{1}{\varepsilon}a^{2}+\varepsilon b^{2}, a, b\geq 0,$
it follows that
\begin{eqnarray*}\label{}
&&\mathbb{\hat{E}}[\sup\limits_{t\in [0,T]}|I_{1}-\int_0^t b_{x}(s,
X_s)Y_sds|^{2}]\\
&\leq &C \int_0^T \mathbb{\hat{E}}[\sup\limits_{r\in
[0,s]}|Z_r^{h}-Y_r|^{2}]ds+C\varepsilon\int_0^T
\mathbb{\hat{E}}[|Y_s|^{4}]ds.\\&&+\frac{C}{\varepsilon}\int_0^T
\mathbb{\hat{E}}[|\widetilde{X}_s-X_s|^{4}+|\widetilde{X}_s-X_s|^{4k+4}+|X_s|^{4k}|\widetilde{X}_s-X_s|^{4}]ds.
\end{eqnarray*}
Then by virtue of Proposition \ref{pr1} and Proposition \ref{pr2},
we obtain
\begin{eqnarray}\label{eq2}
&&\mathbb{\hat{E}}[\sup\limits_{t\in [0,T]}|I_{1}-\int_0^t b_{x}(s,
X_s)Y_sds|^{2}]\nonumber\\
&\leq &C \int_0^T \mathbb{\hat{E}}[\sup\limits_{r\in
[0,s]}|Z_r^{h}-Y_r|^{2}]ds+C\varepsilon+\frac{C}{\varepsilon}
(h^{4}+h^{4k+4}).
\end{eqnarray}
Using  similar arguments, we obtain that for some $m\in\mathbb{N}$,
\begin{eqnarray}\label{eq3}
&&\mathbb{\hat{E}}[\sup\limits_{t\in [0,T]}|I_{2}-\int_0^t
\sigma_{x}(s, X_s)Y_sdB_{s}|^{2}]\nonumber\\
&\leq &C \int_0^T \mathbb{\hat{E}}[\sup\limits_{r\in
[0,s]}|Z_r^{h}-Y_r|^{2}]ds+C\varepsilon+\frac{C}{\varepsilon}
(h^{4}+h^{4m+4}).
\end{eqnarray}
and for some $ n\in\mathbb{N}$,
\begin{eqnarray}\label{eq4}
&&\mathbb{\hat{E}}[\sup\limits_{t\in [0,T]}|I_{3}-\int_0^t h_{x}(s,
X_s)Y_sd\langle
B\rangle_{s}|^{2}]\nonumber\\
&\leq &C \int_0^T \mathbb{\hat{E}}[\sup\limits_{r\in
[0,s]}|Z_r^{h}-Y_r|^{2}]ds+C\varepsilon+\frac{C}{\varepsilon}
(h^{4}+h^{4n+4}).
\end{eqnarray}
Then (\ref{eq6}), (\ref{eq5}), (\ref{eq2}), (\ref{eq3}) and
(\ref{eq4})   yield
\begin{eqnarray*}\label{}
&&\mathbb{\hat{E}}[\sup\limits_{t\in [0,T]}|Z_t^{h}-Y_t|^{2}]\nonumber\\
&\leq &C \int_0^T \mathbb{\hat{E}}[\sup\limits_{r\in
[0,s]}|Z_r^{h}-Y_r|^{2}]ds+C\varepsilon+\frac{C}{\varepsilon}
(h^{4}+h^{4k+4}+h^{4m+4}+h^{4n+4}),
\end{eqnarray*}
and from Gronwall's inequality it follows that
\begin{eqnarray*}\label{}
\mathbb{\hat{E}}[\sup\limits_{t\in [0,T]}|Z_t^{h}-Y_t|^{2}]\leq
C\varepsilon+\frac{C}{\varepsilon}
(h^{4}+h^{4k+4}+h^{4m+4}+h^{4n+4}).
\end{eqnarray*}
Letting $h\rightarrow0$, we get
\begin{eqnarray*}\label{}
\lim\limits_{h\rightarrow0}\mathbb{\hat{E}}[\sup\limits_{t\in
[0,T]}|Z_t^{h}-Y_t|^{2}]\leq C\varepsilon.
\end{eqnarray*}
Therefore,
\begin{eqnarray}\label{eq7}
\lim\limits_{h\rightarrow0}\mathbb{\hat{E}}[\sup\limits_{t\in
[0,T]}|Z_t^{h}-Y_t|^{2}]=0.
\end{eqnarray}
  The proof is complete.
\end{proof}

\begin{remark}\label{re1}
We can check that the following holds true, which we will use in
what follows.
\begin{eqnarray*}\label{}
\lim\limits_{h\rightarrow0}\mathbb{\hat{E}}[\sup\limits_{t\in
[0,T]}|Z_t^{h}-Y_t|^{4}]=0.
\end{eqnarray*}
\end{remark}

\begin{theorem}\label{th2}
Under the conditions of Theorem \ref{th1},  $\frac{\partial
X^{x}_{t}}{\partial x}$ is continuous with respect to $t$ in
$L_{G}^{2}$.
\end{theorem}

\begin{proof}
We use the same notations as those in the proof of Theorem
\ref{th1}. For all $t\in[0,T]$, by the condition of $b,\sigma, h$
and Proposition \ref{pr2}  we have
 \begin{eqnarray*}\label{}
\mathbb{\hat{E}}[|Z_t^{h}|^{2}]&\leq&4+\frac{4}{h^{2}}\mathbb{\hat{E}}[|\int_0^t
[b(s,\widetilde{X}_s)-b(s,X_s)]ds|^{2}]
+\frac{4}{h^{2}}\mathbb{\hat{E}}[|\int_0^t [\sigma(s, \widetilde{X}_s)-\sigma(s, X_s)]dB_{s}|^{2}]\\
&&+\frac{4}{h^{2}}\mathbb{\hat{E}}[|\int_0^t [h(s,
\widetilde{X}_s)-h(s, X_s)]d\langle B \rangle_{s}|^{2}]\\
&\leq&4+\frac{4t}{h^{2}}\int_0^t
\mathbb{\hat{E}}[|b(s,\widetilde{X}_s)-b(s,X_s)|^{2}]ds
+\frac{4}{h^{2}}\int_0^t \mathbb{\hat{E}}[|\sigma(s, \widetilde{X}_s)-\sigma(s, X_s)|^{2}]ds\\
&&+\frac{4t}{h^{2}}\int_0^t \mathbb{\hat{E}}[|h(s,
\widetilde{X}_s)-h(s, X_s)|^{2}]ds\\
&\leq&4+\frac{C}{h^{2}}\int_0^t
\mathbb{\hat{E}}[|\widetilde{X}_s-X_s|^{2}]ds\\
&\leq&C,
\end{eqnarray*}
where $C$ is a constant depending on $T$ and the Lipschitz constant
of $b,\sigma, h$. Therefore, from (\ref{eq7}) we have
 \begin{eqnarray*}\label{}
\mathbb{\hat{E}}[|Y_t|^{2}]\leq
2\mathbb{\hat{E}}[|Y_t-Z_t^{h}|^{2}]+2\mathbb{\hat{E}}[|Z_t^{h}|^{2}]\leq
K<\infty.
\end{eqnarray*}
Without loss of generality, we suppose that $0\leq r \leq t \leq T$.
By (\ref{eq8}) we have
 \begin{eqnarray*}\label{}
\mathbb{\hat{E}}[|Z_t^{h}-Z_r^{h}|^{2}]&\leq&\frac{3}{h^{2}}\mathbb{\hat{E}}[|\int_r^t
[b(s,\widetilde{X}_s)-b(s,X_s)]ds|^{2}]
+\frac{3}{h^{2}}\mathbb{\hat{E}}[|\int_r^t [\sigma(s, \widetilde{X}_s)-\sigma(s, X_s)]dB_{s}|^{2}]\\
&&+\frac{3}{h^{2}}\mathbb{\hat{E}}[|\int_r^t [h(s,
\widetilde{X}_s)-h(s, X_s)]d\langle B \rangle_{s}|^{2}].
\end{eqnarray*}
Then by virtue of Lipschitz condition of $b$ and Proposition
\ref{pr2}, we obtain
 \begin{eqnarray*}\label{}
&&\frac{1}{h^{2}}\mathbb{\hat{E}}[|\int_r^t
[b(s,\widetilde{X}_s)-b(s,X_s)]ds|^{2}]\\
&\leq&\frac{t-r}{h^{2}}\int_r^t
\mathbb{\hat{E}}[|b(s,\widetilde{X}_s)-b(s,X_s)|^{2}]ds \\
&\leq&\frac{C(t-r)}{h^{2}}\int_r^t
\mathbb{\hat{E}}[|\widetilde{X}_s-X_s|^{2}]ds \\
&\leq&C(t-r)^{2},
\end{eqnarray*}
where $C$ is a constant which is independent of $t,s$. Using similar
arguments we obtain
\begin{eqnarray*}\label{}
\frac{1}{h^{2}}\mathbb{\hat{E}}[|\int_r^t [\sigma(s,
\widetilde{X}_s)-\sigma(s, X_s)]dB_{s}|^{2}] \leq C(t-r),
\end{eqnarray*}
and
\begin{eqnarray*}\label{}
\frac{1}{h^{2}}\mathbb{\hat{E}}[|\int_r^t [h(s,
\widetilde{X}_s)-h(s, X_s)]d\langle B \rangle_{s}|^{2}] \leq
C(t-r)^{2}.
\end{eqnarray*}
Then from the above inequalities, we have
 \begin{eqnarray}\label{eq9}
\mathbb{\hat{E}}[|Z_t^{h}-Z_r^{h}|^{2}]\leq C(t-r).
\end{eqnarray}
Therefore, by (\ref{eq7}) and (\ref{eq9}) we have
 \begin{eqnarray*}\label{}
\mathbb{\hat{E}}[|Y_t-Y_s|^{2}]&\leq &
3\mathbb{\hat{E}}[|Y_t-Z_t^{h}|^{2}]+3\mathbb{\hat{E}}[|Z_t^{h}-Z_s^{h}|^{2}]+3\mathbb{\hat{E}}[|Z_s^{h}-Y_s|^{2}]\\
&\leq &
3\mathbb{\hat{E}}[|Y_t-Z_t^{h}|^{2}]+C(t-r)+3\mathbb{\hat{E}}[|Z_s^{h}-Y_s|^{2}].
\end{eqnarray*}
Letting $h\rightarrow 0$ and $t\rightarrow s$, we get the desired
result.  The proof is complete.
\end{proof}\vskip1mm

\begin{theorem}\label{th2}
 Under the conditions of Theorem \ref{th1},  and for $t\in[0,T],$
 if $b_{xx}(t,\cdot),\sigma_{xx}(t,\cdot), h_{xx}(t,\cdot)\in
C_{l,lip}(\mathbb{R})$ and are bounded, then $\frac{\partial
X^{x}_{t}}{\partial x}$ is continuously differentiable in
$L_{G}^{2}$ with respect to $x$. Moreover,
$P_{t}^{x}:=\frac{\partial^{2} X^{x}_{t}}{\partial x^{2}}$ satisfies
the following stochastic differential equation
\begin{eqnarray*}\label{}
P_t^{x}&=&\int_0^t b_{xx}(s, X_s^{x})(Y_s^{x})^{2}ds+\int_0^t
\sigma_{xx}(s, X_s^{x})(Y_s^{x})^{2}dB_{s}+\int_0^t h_{xx}(s,
X_s^{x})(Y_s^{x})^{2}d\langle
B\rangle_{s}\\
&&+\int_0^t b_{x}(s, X_s^{x})P_s^{x}ds+\int_0^t \sigma_{x}(s,
X_s^{x})P_s^{x}dB_{s}+\int_0^t h_{x}(s, X_s^{x})P_s^{x}d\langle
B\rangle_{s}, \ t\in [0,T],
\end{eqnarray*}
where $Y_t^{x}$ is defined in Theorem \ref{th1}.
\end{theorem}

\begin{proof}
Let $h\neq 0$ be small. We use the same notations as Theorem
\ref{th1}. For simplicity, we also put
\begin{eqnarray*}
P_t:=P^{x}_t,\ \widetilde{Y}_t:=Y^{x+h}_t, \
Q^{h}_t:=\frac{\widetilde{Y}_t-Y_t}{h}.
\end{eqnarray*}
Then we have
 \begin{eqnarray*}\label{}
Q_t^{h}&=&\frac{1}{h}\int_0^t [b_{x}(s,
\widetilde{X}_s)\widetilde{Y}_{s}-b_{x}(s,
X_s)Y_s]ds+\frac{1}{h}\int_0^t [\sigma_{x}(s,
\widetilde{X}_s)\widetilde{Y}_{s}-\sigma_{x}(s,
X_s)Y_s]dB_{s}\\
&&+\frac{1}{h}\int_0^t [h_{x}(s,
\widetilde{X}_s)\widetilde{Y}_{s}-h_{x}(s, X_s)Y_s]d\langle B
\rangle_{s}.
\end{eqnarray*}
Since $b_{xx}(t,\cdot),\sigma_{xx}(t,\cdot), h_{xx}(t,\cdot)\in
C_{l,lip}(\mathbb{R})$, for all $t\in[0,T]$,  we have
 \begin{eqnarray}\label{eq12}
Q_t^{h}=I_{1}+I_{2}+I_{3},
\end{eqnarray}
where
\begin{eqnarray*}\label{}
I_{1}&=&\int_0^t b_{x}(s, \widetilde{X}_s) Q_s^{h}ds+\int_0^t
\int_0^1 b_{xx}(s, X_s+\theta(\widetilde{X}_s-X_s))d\theta
Z_s^{h}Y_s ds, \\
I_{2}&=&\int_0^t \sigma_{x}(s, \widetilde{X}_s)
Q_s^{h}dB_{s}+\int_0^t \int_0^1 \sigma_{xx}(s,
X_s+\theta(\widetilde{X}_s-X_s))d\theta Z_s^{h}Y_sdB_{s},\\
I_{3}&=& \int_0^t h_{x}(s, \widetilde{X}_s) Q_s^{h}d\langle B
\rangle_{s}+\int_0^t \int_0^1 h_{xx}(s,
X_s+\theta(\widetilde{X}_s-X_s))d\theta Z_s^{h}Y_sd\langle B
\rangle_{s}.
\end{eqnarray*}
 Since $b_{x}(t,\cdot)\in C_{l,lip}(\mathbb{R})$, $t\in[0,T]$,
and is bounded, we have, for some $k\in\mathbb{N}$,
 \begin{eqnarray*}\label{}
&&\mathbb{\hat{E}}[\sup\limits_{t\in [0,T]}|\int_0^t b_{x}(s,
\widetilde{X}_s) Q_s^{h}ds-\int_0^t b_{x}(s,
X_s)P_sds|^{2}]\\
&\leq &2\mathbb{\hat{E}}[(\int_0^T  |b_{x}(s, \widetilde{X}_s)|
|Q_s^{h}-P_s| ds)^{2}]\\&&+2\mathbb{\hat{E}}[(\int_0^T |b_{x}(s,
\widetilde{X})- b_{x}(s, X_s)| |P_s| ds)^{2}]\\
&\leq &C\mathbb{\hat{E}}[\int_0^T |Q_s^{h}-P_s|^{2}
ds]\\&&+C\mathbb{\hat{E}}[(\int_0^T
(|\widetilde{X}_s-X_s|+|\widetilde{X}_s-X_s|^{k+1}+
|X_s|^{k}|\widetilde{X}_s-X_s|) |P_s| ds)^{2}].
\end{eqnarray*}
For all $\varepsilon>0,$ from
$2ab\leq\frac{1}{\varepsilon}a^{2}+\varepsilon b^{2}, a, b\geq 0,$
it follows that,
\begin{eqnarray*}\label{}
&&\mathbb{\hat{E}}[\sup\limits_{t\in [0,T]}|\int_0^t b_{x}(s,
\widetilde{X}_s) Q_s^{h}ds-\int_0^t b_{x}(s,
X_s)P_sds|^{2}]\\
&\leq &C\mathbb{\hat{E}}[\int_0^T |Q_s^{h}-P_s|^{2}
ds]+C\varepsilon\int_0^T
\mathbb{\hat{E}}[|P_s|^{4}]ds\\&&+\frac{C}{\varepsilon}\int_0^T
\mathbb{\hat{E}}[|\widetilde{X}_s-X_s|^{4}+|\widetilde{X}_s-X_s|^{4k+4}+|X_s|^{4k}|\widetilde{X}_s-X_s|^{4}]ds.
\end{eqnarray*}
Then by virtue of Proposition \ref{pr1} and Proposition \ref{pr2},
we obtain
\begin{eqnarray}\label{eq10}
&&\mathbb{\hat{E}}[\sup\limits_{t\in [0,T]}|\int_0^t b_{x}(s,
\widetilde{X}_s) Q_s^{h}ds-\int_0^t b_{x}(s,
X_s)P_sds|^{2}]\nonumber\\
 &\leq &C\int_0^T \mathbb{\hat{E}}[\sup\limits_{r\in [0,s]}|Q_r^{h}-P_r|^{2}]
ds+C\varepsilon+\frac{C}{\varepsilon} (h^{4}+h^{4k+4}).
\end{eqnarray}
Since $b_{xx}(t,\cdot)\in C_{l,lip}(\mathbb{R})$,  $t\in[0,T]$ and
is bounded, we have,  for some $l\in\mathbb{N}$,
 \begin{eqnarray*}\label{}
&&\mathbb{\hat{E}}[\sup\limits_{t\in [0,T]}|\int_0^t \int_0^1
b_{xx}(s, X_s+\theta(\widetilde{X}_s-X_s))d\theta Z_s^{h}Y_s ds
-\int_0^t b_{xx}(s,X_s)Y^{2}_sds|^{2}]\\
&\leq &2\mathbb{\hat{E}}[(\int_0^T \int_0^1 |b_{xx}(s,
X_s+\theta(\widetilde{X}_s-X_s))|d\theta |Z_s^{h}Y_s-Y^{2}_s| ds
)^{2}]\\&&+2\mathbb{\hat{E}}[(\int_0^T \int_0^1 |b_{xx}(s,
X_s+\theta(\widetilde{X}_s-X_s))-b_{xx}(s,X_s)|d\theta
Y^{2}_sds)^{2}]\\
&\leq & C\mathbb{\hat{E}}[(\int_0^T  |Z_s^{h}Y_s-Y^{2}_s| ds
)^{2}]\\&&+C\mathbb{\hat{E}}[(\int_0^T
(|\widetilde{X}_s-X_s|+|\widetilde{X}_s-X_s|^{l+1}+
|X_s|^{l}|\widetilde{X}_s-X_s|) Y^{2}_sds)^{2}].
\end{eqnarray*}
For all $\varepsilon>0,$ from
$2ab\leq\frac{1}{\varepsilon}a^{2}+\varepsilon b^{2}, a, b\geq 0,$
it follows that
 \begin{eqnarray*}\label{}
&&\mathbb{\hat{E}}[\sup\limits_{t\in [0,T]}|\int_0^t \int_0^1
b_{xx}(s, X_s+\theta(\widetilde{X}_s-X_s))d\theta Z_s^{h}Y_s ds
-\int_0^t b_{xx}(s,X_s)Y^{2}_sds|^{2}]\\&\leq
&\frac{C}{\varepsilon}\mathbb{\hat{E}}[\sup\limits_{t\in [0,T]}
|Z_s^{h}-Y_s| ^{4}]+C\varepsilon\int_0^T
\mathbb{\hat{E}}[|Y_s|^{4}]ds\\&&+\frac{C}{\varepsilon}\int_0^T
\mathbb{\hat{E}}[|\widetilde{X}_s-X_s|^{4}+|\widetilde{X}_s-X_s|^{4l+4}+|X_s|^{4l}|\widetilde{X}_s-X_s|^{4}]ds.
\end{eqnarray*}
Then by means of Proposition \ref{pr1} and Proposition \ref{pr2}, we
obtain
\begin{eqnarray}\label{eq11}
&&\mathbb{\hat{E}}[\sup\limits_{t\in [0,T]}|\int_0^t \int_0^1
b_{xx}(s, X_s+\theta(\widetilde{X}_s-X_s))d\theta Z_s^{h}Y_s ds
-\int_0^t b_{xx}(s,X_s)Y^{2}_sds|^{2}]\nonumber\\&\leq
&\frac{C}{\varepsilon}\mathbb{\hat{E}}[\sup\limits_{t\in [0,T]}
|Z_s^{h}-Y_s| ^{4}]+
 C\varepsilon+\frac{C}{\varepsilon} (h^{4}+h^{4l+4}).
\end{eqnarray}
From (\ref{eq10}) and (\ref{eq11}) it follows that
\begin{eqnarray*}\label{}
&&\mathbb{\hat{E}}[\sup\limits_{t\in [0,T]}|I_{1}-\int_0^t b_{x}(s,
X_s)P_sds-\int_0^t b_{xx}(s,X_s)Y^{2}_sds|^{2}]\nonumber\\
&\leq &C\int_0^T \mathbb{\hat{E}}[\sup\limits_{r\in
[0,s]}|Q_r^{h}-P_r|^{2}] ds+C\varepsilon+\frac{C}{\varepsilon}
(h^{4}+h^{4k+4}+h^{4l+4})+\frac{C}{\varepsilon}\mathbb{\hat{E}}[\sup\limits_{t\in
[0,T]} |Z_s^{h}-Y_s| ^{4}].
\end{eqnarray*}
Using similar arguments we obtain,  for some $m,n\in\mathbb{N}$,
\begin{eqnarray*}\label{}
&&\mathbb{\hat{E}}[\sup\limits_{t\in [0,T]}|I_{2}-\int_0^t
\sigma_{x}(s, X_s)P_sdB_{s}-\int_0^t \sigma_{xx}(s,X_s)Y^{2}_sdB_{s}|^{2}]\nonumber\\
&\leq &C\int_0^T \mathbb{\hat{E}}[\sup\limits_{r\in
[0,s]}|Q_r^{h}-P_r|^{2}]ds+C\varepsilon+\frac{C}{\varepsilon}
(h^{4}+h^{4m+4}+h^{4n+4})+\frac{C}{\varepsilon}\mathbb{\hat{E}}[\sup\limits_{t\in
[0,T]} |Z_s^{h}-Y_s| ^{4}].
\end{eqnarray*}
and for some $p, q\in\mathbb{N}$,
\begin{eqnarray*}\label{}
&&\mathbb{\hat{E}}[\sup\limits_{t\in [0,T]}|I_{3}-\int_0^t h_{x}(s,
X_s)P_sd\langle B\rangle_{s}-\int_0^th_{xx}(s,X_s)Y^{2}_sd\langle
B\rangle_{s}|^{2}]\nonumber\\
&\leq &C\int_0^T \mathbb{\hat{E}}[\sup\limits_{r\in
[0,s]}|Q_r^{h}-P_r|^{2}]ds+C\varepsilon+\frac{C}{\varepsilon}
(h^{4}+h^{4p+4}+h^{4q+4})+\frac{C}{\varepsilon}\mathbb{\hat{E}}[\sup\limits_{t\in
[0,T]} |Z_s^{h}-Y_s| ^{4}].
\end{eqnarray*}
Then  the above inequalities yield
\begin{eqnarray*}\label{}
&&\mathbb{\hat{E}}[\sup\limits_{t\in [0,T]}|Q_t^{h}-P_t|^{2}]\\
&\leq &C\int_0^T \mathbb{\hat{E}}[\sup\limits_{r\in
[0,s]}|Q_r^{h}-P_r|^{2}]ds+\frac{C}{\varepsilon}\mathbb{\hat{E}}[\sup\limits_{t\in
[0,T]} |Z_s^{h}-Y_s| ^{4}]\\
&&+C\varepsilon+\frac{C}{\varepsilon}
(h^{4}+h^{4k+4}+h^{4l+4}+h^{4m+4}+h^{4n+4}+h^{4p+4}+h^{4q+4}),
\end{eqnarray*}
and Gronwall's inequality yields
\begin{eqnarray*}\label{}
\mathbb{\hat{E}}[\sup\limits_{t\in [0,T]}|Q_t^{h}-P_t|^{2}]&\leq&
C\varepsilon+\frac{C}{\varepsilon}\mathbb{\hat{E}}[\sup\limits_{t\in
[0,T]} |Z_s^{h}-Y_s| ^{4}]\\&&+\frac{C}{\varepsilon}
(h^{4}+h^{4k+4}+h^{4l+4}+h^{4m+4}+h^{4n+4}+h^{4p+4}+h^{4q+4}).
\end{eqnarray*}
By Remark \ref{re1}  we get
\begin{eqnarray*}\label{}
\lim\limits_{h\rightarrow0}\mathbb{\hat{E}}[\sup\limits_{t\in
[0,T]}|Q_t^{h}-P_t|^{2}]\leq C\varepsilon.
\end{eqnarray*}
Therefore,
\begin{eqnarray}\label{eq13}
\lim\limits_{h\rightarrow0}\mathbb{\hat{E}}[\sup\limits_{t\in
[0,T]}|Q_t^{h}-P_t|^{2}]=0.
\end{eqnarray}
  The proof is complete.
\end{proof}

\section{ Differentiability of solutions of SDEs  driven by the $G$-Brownian motion with respect to the parameter}
The aim of this section is to study  the differentiability of
solutions of stochastic differential equations  driven by the
$G$-Brownian motion with respect to the parameter.  We consider the
following stochastic differential equation depending on the
parameter:
\begin{equation}\label{equ10}
X_t^{\alpha}=x(\alpha)+\int_0^t b(\alpha, s,
X_s^{\alpha})ds+\int_0^t \sigma(\alpha, s,
X_s^{\alpha})dB_{s}+\int_0^t h(\alpha, s, X_s^{\alpha})d\langle
B\rangle_{s}, \quad  t\in [0,T],
\end{equation}
where the initial condition $\alpha\mapsto x(\alpha):
\mathbb{R}\mapsto \mathbb{R}$ is given. Let us make the following
assumptions:

\begin{list}{}{\setlength{\itemindent}{0cm}}
\item[(H4.1)]  The  functions $x, b, \sigma, h: \mathbb{R}\times[0,T]\times\mathbb{R}\mapsto \mathbb{R}$
 are  Lipschitz in $x$ and $\alpha$, uniformly with respect to $t\in[0,T]$.

\end{list}

\begin{proposition}
Under the assumption (H4.1), $\{X_t^{\alpha}\}$ is continuous in
$L_{G}^{2}$ with respect to $\alpha$.
\end{proposition}

\begin{proof}
Let $\alpha, \beta \in\mathbb{R},$ and $t\in[0,T]$.  From the
stochastic differential equation (\ref{equ10}), Lemma \ref{guji},
the subadditivity of $G$-expectation and (H4.1), we have
\begin{eqnarray*}
\mathbb{\hat{E}}[|X_t^{\alpha}-X_t^{\beta}|^{2}]
 &\leq &4 \Big(|x(\alpha)-x(\beta)|^{2}+\mathbb{\hat{E}}\Big[|\int_0^t
\Big[b(\alpha,X_s^{\alpha})-b(\beta, X_s^{\beta})\Big]ds|^{2}\Big]\\
&&+\mathbb{\hat{E}}\Big[|\int_0^t \Big[\sigma(\alpha,
X_s^{\alpha})-\sigma(\beta, X_s^{\beta})\Big]dB_{s}|^{2}\Big]
\\&&+\mathbb{\hat{E}}\Big[|
\int_0^t\Big[h(\alpha, X_s^{\alpha})- h(\beta,
X_s^{\beta})\Big]d\langle B\rangle_{s}|^{2}\Big]\Big)\\
&\leq &C\Big(|\alpha-\beta|^{2}+\int_0^t
\mathbb{\hat{E}}\Big[|X_s^{\alpha}-X_s^{\beta}|^{2}\Big]ds\Big).
\end{eqnarray*}
Thanks to Gronwall's inequality we  get
\begin{eqnarray}\label{equ12}
\mathbb{\hat{E}}[|X_t^{\alpha}-X_t^{\beta}|^{2}] &\leq
&C|\alpha-\beta|^{2},\ t\in[0,T],
\end{eqnarray}
then we get the desired result.
\end{proof}

\begin{proposition}\label{pr3}
Under  the assumption (H4.1), we have  the following estimate for
the solution $X$ of SDE (\ref{equ10}): for all $p\geq0$ and $r\in
[0,T]$,
$$\mathbb{\hat{E}}[\sup\limits_{t\in[0,r]}|X_{t}^{\alpha}|^{p}]\leq C<+\infty.$$
Here, the constant $C=C(p,r,\alpha,L)$.
\end{proposition}

\begin{proposition}\label{pr4}
 Let us assume (H4.1). Then for every $p\geq2$, there exists a positive constant $C$ such
that
$$\mathbb{\hat{E}}[\sup\limits_{t\in[0,T]}|X_t^{\alpha}-X_{t}^{\beta}|^{p}]
\leq C|\alpha-\beta|^{p}, \ \text{for all} \ t\in [0,T].$$ Here
$C=C(p,T, L).$
\end{proposition}
The proofs of the above propositions are similar to the proof of
Proposition \ref{pr1}. We omit it here.

\begin{theorem}\label{th4}
For $t\in[0,T]$,   if $b_{x}(\cdot, t, \cdot),\sigma_{x}(\cdot, t,
\cdot), h_{x}(\cdot, t, \cdot), b_{\alpha}(\cdot, t,
\cdot),\sigma_{\alpha}(\cdot, t, \cdot), h_{\alpha}(\cdot, t,
\cdot)\in C_{l,lip}(\mathbb{R}^{2})$, $b_{x}(\cdot, t,
\cdot),\sigma_{x}(\cdot, t, \cdot), h_{x}(\cdot, t, \cdot)$ are
bounded and $x(\cdot)$ is differentiable, then $X^{\alpha}_{t}$ is
differentiable in $L_{G}^{2}$ with respect to $\alpha$. Moreover,
for all $ t\in [0,T], Y_{t}^{\alpha}:=\frac{\partial
X^{\alpha}_{t}}{\partial \alpha}$ satisfies the following stochastic
differential equation
\begin{eqnarray*}\label{}
Y_t^{\alpha}&=&x'(\alpha)+\int_0^t b_{x}(\alpha, s,
X_s^{\alpha})Y_s^{\alpha}ds+\int_0^t \sigma_{x}(\alpha, s,
X_s^{\alpha})Y_s^{\alpha}dB_{s}+\int_0^t h_{x}(\alpha, s,
X_s^{\alpha})Y_s^{\alpha}d\langle B\rangle_{s}\\
&&+\int_0^t b_{\alpha}(\alpha, s, X_s^{\alpha})ds+\int_0^t
\sigma_{\alpha}(\alpha, s, X_s^{\alpha})dB_{s}+\int_0^t
h_{\alpha}(\alpha, s, X_s^{\alpha})d\langle B\rangle_{s}.
\end{eqnarray*}
\end{theorem}

\begin{proof}
Let $h\neq 0$ be small. For simplicity, we put
\begin{eqnarray*}
X_t:=X_t^{\alpha},\ Y_t:=Y_t^{\alpha}, \
\widetilde{X}_t:=X_t^{\alpha+h}, \
Z^{h}_t:=\frac{\widetilde{X}_t-X_t}{h}.
\end{eqnarray*}
Then we have
 \begin{eqnarray}\label{e3}
Z_t^{h}&=&\frac{x(\alpha+h)-x(\alpha)}{h}+\frac{1}{h}\int_0^t
[b(\alpha+h, s, \widetilde{X}_s)-b(\alpha, s, X_s)]ds
\nonumber\\&&+\frac{1}{h}\int_0^t [\sigma(\alpha+h, s,
\widetilde{X}_s)-\sigma(\alpha, s,
X_s)]dB_{s}\nonumber\\&&+\frac{1}{h}\int_0^t [h(\alpha+h, s,
\widetilde{X}_s)-h(\alpha, s, X_s)]d\langle B
\rangle_{s}\\
&=&\frac{x(\alpha+h)-x(\alpha)}{h}+I_{1}+I_{2}+I_{3},\nonumber
\end{eqnarray}
where
 \begin{eqnarray*}\label{}
I_{1}&=&\frac{1}{h}\int_0^t [b(\alpha+h, s,
\widetilde{X}_s)-b(\alpha, s,
\widetilde{X}_s)]ds+\frac{1}{h}\int_0^t [b(\alpha, s,
\widetilde{X}_s)-b(\alpha, s, X_s)]ds,\\
I_{2}&=&\frac{1}{h}\int_0^t [\sigma(\alpha+h, s,
\widetilde{X}_s)-\sigma(\alpha, s,
\widetilde{X}_s)]dB_{s}+\frac{1}{h}\int_0^t [\sigma(\alpha, s,
\widetilde{X}_s)-\sigma(\alpha, s, X_s)]dB_{s},\\
I_{3}&=&\frac{1}{h}\int_0^t [h(\alpha+h, s,
\widetilde{X}_s)-h(\alpha, s, \widetilde{X}_s)]d\langle B
\rangle_{s}+\frac{1}{h}\int_0^t [h(\alpha, s,
\widetilde{X}_s)-h(\alpha, s, X_s)]d\langle B \rangle_{s}.
\end{eqnarray*}
Since  $b_{x}(\cdot, t, \cdot), b_{\alpha}(\cdot, t, \cdot)\in
C_{l,lip}(\mathbb{R}^{2})$, $t\in[0,T]$, we have
 \begin{eqnarray*}\label{}
I_{1}&=&\int_0^t \int_0^1 b_{\alpha}(\alpha+\theta h, s,
\widetilde{X}_s)d\theta ds+\int_0^t \int_0^1 b_{x}(\alpha, s,
X_s+\theta(\widetilde{X}_s-X_s))d\theta Z_s^{h} ds.
\end{eqnarray*}
Therefore, for some $k\in\mathbb{N}$,
 \begin{eqnarray*}\label{}
&&\mathbb{\hat{E}}[\sup\limits_{t\in [0,T]}|\int_0^t \int_0^1
b_{\alpha}(\alpha+\theta h, s, \widetilde{X}_s)d\theta ds-\int_0^t
b_{\alpha}(\alpha, s, X_s)ds|^{2}]\\
&\leq &C\mathbb{\hat{E}}[\int_0^T \int_0^1| b_{\alpha}(\alpha+\theta
h, s, \widetilde{X}_s)-b_{\alpha}(\alpha, s,
\widetilde{X}_s)|^{2}d\theta ds]\\&&+C\mathbb{\hat{E}}[\int_0^T |
b_{\alpha}(\alpha, s, \widetilde{X}_s)-b_{\alpha}(\alpha, s,
X_s)|^{2}ds]\\
&\leq&C(h^{2}+h^{2k+2})\\&&+C\int_0^T
\mathbb{\hat{E}}[(|\widetilde{X}_s-X_s|+|\widetilde{X}_s-X_s|^{k+1}+
|X_s|^{k}|\widetilde{X}_s-X_s|)^{2}]ds,
\end{eqnarray*}
and by Proposition \ref{pr3} and Proposition \ref{pr4}  we obtain
 \begin{eqnarray*}\label{}
&&\mathbb{\hat{E}}[\sup\limits_{t\in [0,T]}|\int_0^t \int_0^1
b_{\alpha}(\alpha+\theta h, s, \widetilde{X}_s)d\theta ds-\int_0^t
b_{\alpha}(\alpha, s, X_s)ds|^{2}]\\
&\leq&C(h^{2}+h^{2k+2}).
\end{eqnarray*}
Using the  argument similar to the proof of Theorem \ref{th1}, we
get for some $l\in\mathbb{N}$,
\begin{eqnarray*}\label{}
&&\mathbb{\hat{E}}[\sup\limits_{t\in [0,T]}|\int_0^t \int_0^1
b_{x}(\alpha, s, X_s+\theta(\widetilde{X}_s-X_s))d\theta Z_s^{h}
ds-\int_0^t b_{x}(\alpha, s, X_s)Y_s ds|^{2}]\\
&\leq &C \int_0^T \mathbb{\hat{E}}[\sup\limits_{r\in
[0,s]}|Z_r^{h}-Y_r|^{2}]ds+C\varepsilon+\frac{C}{\varepsilon}
(h^{4}+h^{4l+4}).
\end{eqnarray*}
Consequently,
\begin{eqnarray*}\label{}
&&\mathbb{\hat{E}}[\sup\limits_{t\in [0,T]}|I_{1}-\int_0^t
b_{\alpha}(\alpha, s, X_s)ds-\int_0^t b_{x}(\alpha, s, X_s)Y_s ds|^{2}]\\
&\leq &C \int_0^T \mathbb{\hat{E}}[\sup\limits_{r\in
[0,s]}|Z_r^{h}-Y_r|^{2}]ds+C\varepsilon+\frac{C}{\varepsilon}
(h^{4}+h^{4l+4})+C(h^{2}+h^{2l+2}).
\end{eqnarray*}
By similar arguments, we have for some $m,n\in\mathbb{N}$,
\begin{eqnarray*}\label{}
&&\mathbb{\hat{E}}[\sup\limits_{t\in [0,T]}|I_{2}-\int_0^t
\sigma_{x}(s, X_s)Y_sdB_{s}-\int_0^t
\sigma_{\alpha}(\alpha, s, X_s)dB_{s}|^{2}]\nonumber\\
&\leq &C \int_0^T \mathbb{\hat{E}}[\sup\limits_{r\in
[0,s]}|Z_r^{h}-Y_r|^{2}]ds+C\varepsilon+\frac{C}{\varepsilon}
(h^{4}+h^{4n+4})+C(h^{2}+h^{2m+2}).
\end{eqnarray*}
and for some $p,q\in\mathbb{N}$,
\begin{eqnarray*}\label{}
&&\mathbb{\hat{E}}[\sup\limits_{t\in [0,T]}|I_{3}-\int_0^t h_{x}(s,
X_s)Y_sd\langle B\rangle_{s}-\int_0^t
h_{\alpha}(\alpha, s, X_s)d\langle B\rangle_{s}|^{2}]\\
&\leq &C \int_0^T \mathbb{\hat{E}}[\sup\limits_{r\in
[0,s]}|Z_r^{h}-Y_r|^{2}]ds+C\varepsilon+\frac{C}{\varepsilon}
(h^{4}+h^{4q+4})+C(h^{2}+h^{2p+2}).
\end{eqnarray*}
Then the above inequalities   yield
\begin{eqnarray*}\label{}
&&\mathbb{\hat{E}}[\sup\limits_{t\in [0,T]}|Z_t^{h}-Y_t|^{2}]\nonumber\\
&\leq &C \int_0^T \mathbb{\hat{E}}[\sup\limits_{r\in
[0,s]}|Z_r^{h}-Y_r|^{2}]ds+C\varepsilon+\frac{C}{\varepsilon}
(h^{4}+h^{4l+4}+h^{4n+4}+h^{4q+4})\\&&+C(h^{2}+h^{2k+2}+h^{2m+2}+h^{2p+2}).
\end{eqnarray*}
Therefore, from Gronwall's inequality it follows that
\begin{eqnarray*}\label{}
\mathbb{\hat{E}}[\sup\limits_{t\in [0,T]}|Z_t^{h}-Y_t|^{2}]&\leq&
C\varepsilon+\frac{C}{\varepsilon}
(h^{4}+h^{4l+4}+h^{4n+4}+h^{4q+4})\\&&+C(h^{2}+h^{2k+2}+h^{2m+2}+h^{2p+2}).
\end{eqnarray*}
Letting $h\rightarrow0$, we get
\begin{eqnarray*}\label{}
\lim\limits_{h\rightarrow0}\mathbb{\hat{E}}[\sup\limits_{t\in
[0,T]}|Z_t^{h}-Y_t|^{2}]\leq C\varepsilon.
\end{eqnarray*}
Therefore,
\begin{eqnarray}\label{eq117}
\lim\limits_{h\rightarrow0}\mathbb{\hat{E}}[\sup\limits_{t\in
[0,T]}|Z_t^{h}-Y_t|^{2}]=0.
\end{eqnarray}
  The proof is complete.
\end{proof}

\begin{remark}\label{re2}
We can check that the following holds true in the proof of the above
theorem, which we will use in what follows.
\begin{eqnarray*}\label{}
\lim\limits_{h\rightarrow0}\mathbb{\hat{E}}[\sup\limits_{t\in
[0,T]}|Z_t^{h}-Y_t|^{4}]=0.
\end{eqnarray*}
\end{remark}

\begin{proposition}\label{}
Under the conditions of Theorem \ref{th4}, if $x(\cdot)$ is
continuously differentiable, then $\frac{\partial
X^{\alpha}_{t}}{\partial \alpha}$ is continuous with respect to $t$
in $L_{G}^{2}$.
\end{proposition}

\begin{proof}
We use the same notations as those in the proof of Theorem
\ref{th4}. For all $t\in[0,T]$, by the condition of $b,\sigma, h$
and Proposition \ref{pr4},  we have
 \begin{eqnarray*}\label{}
\mathbb{\hat{E}}[|Z_t^{h}|^{2}]&\leq&4+\frac{4}{h^{2}}\mathbb{\hat{E}}[|\int_0^t
[b(\alpha+h, s,\widetilde{X}_s)-b(\alpha, s, X_s)]ds|^{2}]\\
&&+\frac{4}{h^{2}}\mathbb{\hat{E}}[|\int_0^t [\sigma(\alpha+h, s, \widetilde{X}_s)-\sigma(\alpha, s, X_s)]dB_{s}|^{2}]\\
&&+\frac{4}{h^{2}}\mathbb{\hat{E}}[|\int_0^t [h(\alpha+h, s,
\widetilde{X}_s)-h(\alpha, s, X_s)]d\langle B \rangle_{s}|^{2}]\\
&\leq&4+\frac{4t}{h^{2}}\int_0^t \mathbb{\hat{E}}[b(\alpha+h,
s,\widetilde{X}_s)-b(\alpha, s,X_s)]^{2}]ds\\
&&+\frac{4}{h^{2}}\int_0^t \mathbb{\hat{E}}[|\sigma(\alpha+h, s, \widetilde{X}_s)-\sigma(\alpha, s, X_s)|^{2}]ds\\
&&+\frac{4t}{h^{2}}\int_0^t \mathbb{\hat{E}}[|h(\alpha+h, s,
\widetilde{X}_s)-h(\alpha, s, X_s)|^{2}]ds\\
&\leq&4+\frac{C}{h^{2}}\int_0^t
\mathbb{\hat{E}}[|\widetilde{X}_s-X_s|^{2}]ds\\
&\leq&C,
\end{eqnarray*}
where $C$ is a constant depending on $T$ and Lipschitz constant of
$b,\sigma, h$. Therefore, from (\ref{eq117}) we have
 \begin{eqnarray*}\label{}
\mathbb{\hat{E}}[|Y_t|^{2}]\leq
2\mathbb{\hat{E}}[|Y_t-Z_t^{h}|^{2}]+2\mathbb{\hat{E}}[|Z_t^{h}|^{2}]\leq
C<\infty.
\end{eqnarray*}
Without loss of generality, we suppose that $0\leq r \leq t \leq T$.
By (\ref{e3}) we have
 \begin{eqnarray*}\label{}
\mathbb{\hat{E}}[|Z_t^{h}-Z_r^{h}|^{2}]&\leq&\frac{3}{h^{2}}\mathbb{\hat{E}}[|\int_r^t
[b(\alpha+h, s,\widetilde{X}_s)-b(\alpha, s, X_s)]ds|^{2}]\\
&&+\frac{3}{h^{2}}\mathbb{\hat{E}}[|\int_r^t [\sigma(\alpha+h, s, \widetilde{X}_s)-\sigma(\alpha, s, X_s)]dB_{s}|^{2}]\\
&&+\frac{3}{h^{2}}\mathbb{\hat{E}}[|\int_r^t [h(\alpha+h, s,
\widetilde{X}_s)-h(\alpha, s, X_s)]d\langle B \rangle_{s}|^{2}].
\end{eqnarray*}
Then by virtue of Lipschitz condition of $b$ and Proposition
\ref{pr4}, we obtain
 \begin{eqnarray*}\label{}
&&\frac{1}{h^{2}}\mathbb{\hat{E}}[|\int_r^t
[b(\alpha+h, s,\widetilde{X}_s)-b(\alpha, s, X_s)]ds|^{2}]\\
&\leq&\frac{t-r}{h^{2}}\int_r^t
\mathbb{\hat{E}}[|b(\alpha+h, s,\widetilde{X}_s)-b(\alpha, s, X_s)|^{2}]ds \\
&\leq&\frac{C(t-r)}{h^{2}}\int_r^t
\mathbb{\hat{E}}[|\widetilde{X}_s-X_s|^{2}+h^{2}]ds \\
&\leq&C(t-r)^{2},
\end{eqnarray*}
where $C$ is a constant which is independent of $t,r$. By a similar
argument we obtain
\begin{eqnarray*}\label{}
\frac{1}{h^{2}}\mathbb{\hat{E}}[|\int_r^t [\sigma(\alpha+h, s,
\widetilde{X}_s)-\sigma(\alpha, s, X_s)]dB_{s}|^{2}] \leq C(t-r),
\end{eqnarray*}
and
\begin{eqnarray*}\label{}
\frac{1}{h^{2}}\mathbb{\hat{E}}[|\int_r^t [h(\alpha+h, s,
\widetilde{X}_s)-h(\alpha, s, X_s)]d\langle B \rangle_{s}|^{2}] \leq
C(t-r)^{2}.
\end{eqnarray*}
Then the above inequalities yield
 \begin{eqnarray}\label{eq119}
\mathbb{\hat{E}}[|Z_t^{h}-Z_r^{h}|^{2}]\leq K(t-r).
\end{eqnarray}
Therefore, by (\ref{eq117}) and (\ref{eq119}) we have
 \begin{eqnarray*}\label{}
\mathbb{\hat{E}}[|Y_t-Y_s|^{2}]&\leq &
3\mathbb{\hat{E}}[|Y_t-Z_t^{h}|^{2}]+3\mathbb{\hat{E}}[|Z_t^{h}-Z_s^{h}|^{2}]+3\mathbb{\hat{E}}[|Z_s^{h}-Y_s|^{2}]\\
&\rightarrow& 0,
\end{eqnarray*}
as $t\rightarrow s$.
 The proof is complete.
\end{proof}

\begin{proposition}\label{pr5}
 Under the conditions of Theorem \ref{th4}, if $x(\cdot)$ is
 continuously differentiable, then
\begin{eqnarray*}\label{}
\lim\limits_{h\rightarrow0}\mathbb{\hat{E}}[\sup\limits_{t\in[0,T]}|Y_t^{\alpha}-Y_t^{\alpha+h}|^{2}]=0.
\end{eqnarray*}
\end{proposition}

\begin{proof}
Let $h\neq 0$ be small. For simplicity, we put
\begin{eqnarray*}
X_t:=X_t^{\alpha},\  \widetilde{X}_t:=X_t^{\alpha+h},\
Y_t:=Y_t^{\alpha}, \ \widetilde{Y}_t:=Y_t^{\alpha+h}.
\end{eqnarray*}
Then we have
 \begin{eqnarray*}\label{}
\widetilde{Y}_t-Y_t&=&x'(\alpha+h)-x'(\alpha)+I_{1}+I_{2}+I_{3},
\end{eqnarray*}
where
 \begin{eqnarray*}\label{}
I_{1}&=&\int_0^t [b_{x}(\alpha+h,
s,\widetilde{X}_s)\widetilde{Y}_{s}-b_{x}(\alpha, s,
X_s)Y_s]ds+\int_0^t [b_{\alpha}(\alpha+h, s,
\widetilde{X}_s)-b_{\alpha}(\alpha, s, X_s)]ds,\\
I_{2}&=&\int_0^t [\sigma_{x}(\alpha+h, s,
\widetilde{X}_s)\widetilde{Y}_{s}-\sigma_{x}(\alpha, s,
X_s)Y_s]dB_{s}+\int_0^t [\sigma_{\alpha}(\alpha+h, s,
\widetilde{X}_s)-\sigma_{\alpha}(\alpha, s, X_s)]dB_{s},\\
I_{3}&=&\int_0^t [h_{x}(\alpha+h, s,
\widetilde{X}_s)\widetilde{Y}_{s}-h_{x}(\alpha, s, X_s)Y_s]d\langle
B \rangle_{s}+\int_0^t [h_{\alpha}(\alpha+h, s,
\widetilde{X}_s)-h_{\alpha}\alpha, s, X_s)]d\langle B \rangle_{s}.
\end{eqnarray*}
For all $\varepsilon>0,$ from
$2ab\leq\frac{1}{\varepsilon}a^{2}+\varepsilon b^{2}, a, b\geq 0,$
it follows that for some $k\in\mathbb{N}$,
\begin{eqnarray*}\label{}
&&\mathbb{\hat{E}}[\sup\limits_{t\in [0,T]}|\int_0^t
[b_{x}(\alpha+h, s,\widetilde{X}_s)\widetilde{Y}_{s}-b_{x}(\alpha,
s, X_s)Y_s]ds|^{2}]\\
&\leq &2\mathbb{\hat{E}}[|\int_0^T [b_{x}(\alpha+h,
s,\widetilde{X}_s)(\widetilde{Y}_{s}-Y_s)]ds|^{2}]\\&&+2\mathbb{\hat{E}}[|\int_0^T
[b_{x}(\alpha+h, s,\widetilde{X}_s)-b_{x}(\alpha,
s, X_s)]Y_sds|^{2}]\\
&\leq &C\mathbb{\hat{E}}\Big[\int_0^T |\widetilde{Y}_{s}-Y_s|^{2}
ds\Big]+C(h^{2}+h^{2k+2})\int_0^T\mathbb{\hat{E}}[|Y_s|^{2}]ds+C\varepsilon\int_0^T
\mathbb{\hat{E}}[|Y_s|^{4}]ds\\&&+\frac{C}{\varepsilon}\int_0^T
\mathbb{\hat{E}}[|\widetilde{X}_s-X_s|^{4}+|\widetilde{X}_s-X_s|^{4k+4}+|X_s|^{4k}|\widetilde{X}_s-X_s|^{4}]ds.
\end{eqnarray*}
Then by virtue of Proposition \ref{pr3} and Proposition \ref{pr4},
we obtain
\begin{eqnarray*}\label{}
&&\mathbb{\hat{E}}[\sup\limits_{t\in [0,T]}|\int_0^t
[b_{x}(\alpha+h, s,\widetilde{X}_s)-b_{x}(\alpha, s, X_s)]ds|^{2}]\\
&\leq &C\mathbb{\hat{E}}[\int_0^T |\widetilde{Y}_{s}-Y_s|^{2}
ds]+C(h^{2}+h^{2k+2})+C\varepsilon+\frac{C}{\varepsilon}
(h^{4}+h^{4k+4}).
\end{eqnarray*}
On the other hand, for some $l\in\mathbb{N}$,
\begin{eqnarray*}\label{}
&&\mathbb{\hat{E}}[\sup\limits_{t\in [0,T]}|\int_0^t
[b_{\alpha}(\alpha+h, s,\widetilde{X}_s)-b_{\alpha}(\alpha, s, X_s)]ds|^{2}]\\
&\leq &C(h^{2}+h^{2l+2})+C\int_0^T
\mathbb{\hat{E}}[|\widetilde{X}_s-X_s|^{2}+|\widetilde{X}_s-X_s|^{2l+2}+|X_s|^{2l}|\widetilde{X}_s-X_s|^{2}]ds\\
&\leq &C(h^{2}+h^{2l+2}).
\end{eqnarray*}
Therefore,
\begin{eqnarray*}\label{}
\mathbb{\hat{E}}[\sup\limits_{t\in [0,T]}|I_{1}|^{2}]\leq
C\mathbb{\hat{E}}[\int_0^T |\widetilde{Y}_{s}-Y_s|^{2}
ds]+C(h^{2}+h^{2k+2}+h^{2l+2})+C\varepsilon+\frac{C}{\varepsilon}
(h^{4}+h^{4k+4}).
\end{eqnarray*}
Similarly, we have for some $m,n\in\mathbb{N}$,
\begin{eqnarray*}\label{}
\mathbb{\hat{E}}[\sup\limits_{t\in [0,T]}|I_{2}|^{2}]\leq
C\mathbb{\hat{E}}[\int_0^T |\widetilde{Y}_{s}-Y_s|^{2}
ds]+C(h^{2}+h^{2m+2}+h^{2n+2})+C\varepsilon+\frac{C}{\varepsilon}
(h^{4}+h^{4m+4}),
\end{eqnarray*}
and for some $p,q\in\mathbb{N}$,
\begin{eqnarray*}\label{}
\mathbb{\hat{E}}[\sup\limits_{t\in [0,T]}|I_{2}|^{2}]\leq
C\mathbb{\hat{E}}[\int_0^T |\widetilde{Y}_{s}-Y_s|^{2}
ds]+C(h^{2}+h^{2p+2}+h^{2q+2})+C\varepsilon+\frac{C}{\varepsilon}
(h^{4}+h^{4p+4}).
\end{eqnarray*}
Thus,
\begin{eqnarray*}\label{}
\mathbb{\hat{E}}[\sup\limits_{t\in
[0,T]}|\widetilde{Y}_t-Y_t|^{2}]&\leq& C\mathbb{\hat{E}}[\int_0^T
\sup\limits_{r\in [0,s]}|\widetilde{Y}_{r}-Y_r|^{2}
ds]+\frac{C}{\varepsilon}
(h^{4}+h^{4k+4}+h^{4m+4}+h^{4p+4})\\
&&+C(h^{2}+h^{2k+2}+h^{2l+2}+h^{2m+2}+h^{2n+2}+h^{2p+2}+h^{2q+2})\\
&&+C\varepsilon+|x'(\alpha+h)-x'(\alpha)|^{2}.
\end{eqnarray*}
Due to  Gronwall's inequality we get
\begin{eqnarray*}\label{}
\mathbb{\hat{E}}[\sup\limits_{t\in
[0,T]}|\widetilde{Y}_t-Y_t|^{2}]&\leq& \frac{C}{\varepsilon}
(h^{4}+h^{4k+4}+h^{4m+4}+h^{4p+4})+C\varepsilon+|x'(\alpha+h)-x'(\alpha)|^{2}\\
&&+C(h^{2}+h^{2k+2}+h^{2l+2}+h^{2m+2}+h^{2n+2}+h^{2p+2}+h^{2q+2}).
\end{eqnarray*}
Letting $h\rightarrow0$, we get
\begin{eqnarray*}\label{}
\lim\limits_{h\rightarrow0}\mathbb{\hat{E}}[\sup\limits_{t\in
[0,T]}|\widetilde{Y}_t-Y_t|^{2}]\leq C\varepsilon.
\end{eqnarray*}
Consequently,
\begin{eqnarray*}\label{}
\lim\limits_{h\rightarrow0}\mathbb{\hat{E}}[\sup\limits_{t\in
[0,T]}|\widetilde{Y}_t-Y_t|^{2}]=0.
\end{eqnarray*}
 The proof is complete.
\end{proof}

\begin{theorem}\label{th2}
 Under the condition of Theorem \ref{th4}, $x(\cdot)$ is twice differentiable and for all $t\in[0,T]$,
 $b_{xx}(\cdot,t,\cdot),  b_{x\alpha}(\cdot,t,\cdot), \sigma_{xx}(\cdot,t,\cdot),
 \sigma_{x\alpha}(\cdot,t,\cdot), h_{xx}(\cdot,t,\cdot), h_{x\alpha}(\cdot,t,\cdot)\in
C_{l,lip}(\mathbb{R}^{2})$ and are bounded, then $\frac{\partial
X^{\alpha}_{t}}{\partial \alpha}$ is continuously differentiable in
$L_{G}^{2}$ with respect to $\alpha$. Moreover,
$P_{t}^{\alpha}:=\frac{\partial^{2} X^{\alpha}_{t}}{\partial
\alpha^{2}}$ satisfies  the following stochastic differential
equation
\begin{eqnarray}\label{e4}
P_t^{\alpha}&=&x''(\alpha)+\int_0^t b_{xx}(\alpha, s,
X_s^{\alpha})(Y_s^{\alpha})^{2}ds+\int_0^t \sigma_{xx}(\alpha, s,
X_s^{\alpha})(Y_s^{\alpha})^{2}dB_{s}\nonumber\\
&&+\int_0^t h_{xx}(\alpha, s,
X_s^{\alpha})(Y_s^{\alpha})^{2}d\langle B\rangle_{s}+\int_0^t
b_{x}(\alpha, s, X_s^{\alpha})P_s^{\alpha}ds+\int_0^t
\sigma_{x}(\alpha, s, X_s^{\alpha})P_s^{\alpha}dB_{s}\nonumber
\\&&+\int_0^t
h_{x}(\alpha, s, X_s^{\alpha})P_s^{\alpha}d\langle
B\rangle_{s}+2\int_0^t b_{x\alpha}(\alpha, s,
X_s^{\alpha})Y_s^{\alpha}ds \nonumber\\
&&+2\int_0^t \sigma_{x\alpha}(\alpha, s,
X_s^{\alpha})Y_s^{\alpha}dB_{s}+2\int_0^t h_{x\alpha}(\alpha, s,
X_s^{\alpha})Y_s^{\alpha}d\langle
B\rangle_{s}\nonumber\\
&&+\int_0^t b_{\alpha\alpha}(\alpha, s, X_s^{\alpha})ds+\int_0^t
\sigma_{\alpha\alpha}(\alpha, s, X_s^{\alpha})dB_{s}+\int_0^t
h_{\alpha\alpha}(\alpha, s, X_s^{\alpha})d\langle B\rangle_{s},
\end{eqnarray}
where $ Y_t^{\alpha}$ is defined in Theorem \ref{th4}.
\end{theorem}

\begin{proof}
Let $h\neq 0$ be small. We use the same notations as Theorem
\ref{th4}. For simplicity, we also put
\begin{eqnarray*}
P_t:=P^{\alpha}_t, \ Q^{h}_t:=\frac{\widetilde{Y}_t-Y_t}{h}.
\end{eqnarray*}
Then we have
 \begin{eqnarray*}\label{}
Q_t^{h}&=&\frac{x(\alpha+h)-x(\alpha)}{h}+I_{1}+I_{2}+I_{3},
\end{eqnarray*}
where
 \begin{eqnarray*}\label{}
I_{1}&=&\frac{1}{h}\int_0^t [b_{x}(\alpha+h, s,
\widetilde{X}_s)\widetilde{Y}_{s}-b_{x}(\alpha, s,
X_s)Y_s]ds\\&&+\frac{1}{h}\int_0^t [b_{\alpha}(\alpha+h, s,
\widetilde{X}_s)-b_{\alpha}(\alpha, s, X_s)]ds,\\
I_{2}&=&\frac{1}{h}\int_0^t [\sigma_{x}(\alpha+h, s,
\widetilde{X}_s)\widetilde{Y}_{s}-\sigma_{x}(\alpha, s,
X_s)Y_s]dB_{s}\\
&&+\frac{1}{h}\int_0^t [\sigma_{\alpha}(\alpha+h, s,
\widetilde{X}_s)-\sigma_{\alpha}(\alpha, s, X_s)]dB_{s},\\
I_{3}&=&\frac{1}{h}\int_0^t [h_{x}(\alpha+h, s,
\widetilde{X}_s)\widetilde{Y}_{s}-h_{x}(s, X_s)Y_s]d\langle B
\rangle_{s}\\
&&+\frac{1}{h}\int_0^t [h_{\alpha}(\alpha+h, s,
\widetilde{X}_s)-h_{\alpha}(\alpha, s, X_s)]d\langle B \rangle_{s}.
\end{eqnarray*}
We divide the proof into several steps.\vskip1mm

\noindent  Step 1:
 Since  $b_{x\alpha}(\cdot,t,\cdot)\in
C_{l,lip}(\mathbb{R}^{2})$, for all $t\in[0,T]$, we have
 \begin{eqnarray*}\label{}
\frac{1}{h}\int_0^t [b_{x}(\alpha+h, s,
\widetilde{X}_s)\widetilde{Y}_{s}-b_{x}(\alpha, s,
\widetilde{X}_s)\widetilde{Y}_s]ds&=&\int_0^t \int_0^1
b_{x\alpha}(\alpha+\theta h, s, \widetilde{X}_s)d\theta
\widetilde{Y}_s ds.
\end{eqnarray*}
Then for some $k\in\mathbb{N}$,
 \begin{eqnarray*}\label{}
&&\mathbb{\hat{E}}[\sup\limits_{t\in [0,T]}|\int_0^t \int_0^1
b_{x\alpha}(\alpha+\theta h, s, \widetilde{X}_s)d\theta
\widetilde{Y}_s ds-\int_0^t b_{x\alpha}(\alpha, s,
X_s)Y_sds|^{2}]\\
&\leq&2\mathbb{\hat{E}}[(\int_0^T \int_0^1
|b_{x\alpha}(\alpha+\theta h, s, \widetilde{X}_s)|d\theta
|\widetilde{Y}_s -Y_s|ds)^{2}]\\&&+2\mathbb{\hat{E}}[(\int_0^T
\int_0^1 |b_{x\alpha}(\alpha+\theta h, s, \widetilde{X}_s)-
b_{x\alpha}(\alpha, s, X_s)|d\theta
|Y_s|ds)^{2}]\\
&\leq &C\mathbb{\hat{E}}[\int_0^T|\widetilde{Y}_s
-Y_s|^{2}ds]+C(h^{2}+h^{2k+2})\int_0^T
\mathbb{\hat{E}}[|Y_s|^{2}]ds\\&&+C\mathbb{\hat{E}}[(\int_0^T
(|\widetilde{X}_s-X_s|+|\widetilde{X}_s-X_s|^{k+1}+
|X_s|^{k}|\widetilde{X}_s-X_s|) |Y_s| ds)^{2}].
\end{eqnarray*}
For all $\varepsilon>0,$ from
$2ab\leq\frac{1}{\varepsilon}a^{2}+\varepsilon b^{2}, a, b\geq 0,$
Proposition \ref{pr3} and Proposition \ref{pr4}, it follows that
\begin{eqnarray*}\label{}
&&\mathbb{\hat{E}}[\sup\limits_{t\in [0,T]}|\int_0^t \int_0^1
b_{x\alpha}(\alpha+\theta h, s, \widetilde{X}_s)d\theta
\widetilde{Y}_s ds-\int_0^t b_{x\alpha}(\alpha, s,
X_s)Y_sds|^{2}]\\
&\leq &C\mathbb{\hat{E}}[\sup\limits_{s\in[0,T]}|\widetilde{Y}_s
-Y_s|^{2}]+C(h^{2}+h^{2k+2})+C\varepsilon+\frac{C}{\varepsilon}
(h^{4}+h^{4k+4}).
\end{eqnarray*}
Step 2: Since
 \begin{eqnarray*}\label{}
\frac{1}{h}\int_0^t [b_{x}(\alpha, s,
\widetilde{X}_s)\widetilde{Y}_{s}-b_{x}(\alpha, s,
\widetilde{X}_s)Y_s]ds=\int_0^t b_{x}(\alpha, s,
\widetilde{X}_s)Q_s^{h} ds,
\end{eqnarray*}
then we have for some $l\in\mathbb{N}$,
\begin{eqnarray*}\label{}
&&\mathbb{\hat{E}}[\sup\limits_{t\in [0,T]}|\int_0^t b_{x}(\alpha,s,
\widetilde{X}_s) Q_s^{h}ds-\int_0^t b_{x}(\alpha, s,
X_s)P_sds|^{2}]\\
&\leq&C\mathbb{\hat{E}}[\int_0^T |b_{x}(\alpha,s,
\widetilde{X}_s)|^{2}
|Q_s^{h}-P_s|^{2}ds]\\&&+C\mathbb{\hat{E}}[\int_0^T |b_{x}(\alpha,s,
\widetilde{X}_s) - b_{x}(\alpha, s, X_s)|^{2}|P_s|^{2}ds]\\
&\leq &C\mathbb{\hat{E}}[\int_0^T |Q_s^{h}-P_s|^{2}
ds]+C\varepsilon\int_0^T
\mathbb{\hat{E}}[|P_s|^{4}]ds\\&&+\frac{C}{\varepsilon}\int_0^T
\mathbb{\hat{E}}[|\widetilde{X}_s-X_s|^{4}+|\widetilde{X}_s-X_s|^{4l+4}+|X_s|^{4l}|\widetilde{X}_s-X_s|^{4}]ds.
\end{eqnarray*}
Thus, by virtue of Proposition \ref{pr3} and Proposition \ref{pr4},
we obtain
\begin{eqnarray*}\label{}
&&\mathbb{\hat{E}}[\sup\limits_{t\in [0,T]}|\int_0^t b_{x}(\alpha,s,
\widetilde{X}_s) Q_s^{h}ds-\int_0^t b_{x}(\alpha, s,
X_s)P_sds|^{2}]\nonumber\\
 &\leq &C\int_0^T \mathbb{\hat{E}}[\sup\limits_{r\in [0,s]}|Q_r^{h}-P_r|^{2}]
ds+C\varepsilon+\frac{C}{\varepsilon} (h^{4}+h^{4l+4}).
\end{eqnarray*}
Step 3: Since
 \begin{eqnarray*}\label{}
\frac{1}{h}\int_0^t [b_{x}(\alpha, s,
\widetilde{X}_s)Y_{s}-b_{x}(\alpha, s, X_s)Y_s]ds=\int_0^t\int_0^1
b_{xx}(\alpha, s, X_s+\theta(\widetilde{X}_s-X_s))d\theta Y_s
Z_s^{h} ds,
\end{eqnarray*}
then we have for some $m\in\mathbb{N}$,
 \begin{eqnarray*}\label{}
&&\mathbb{\hat{E}}[\sup\limits_{t\in [0,T]}|\int_0^t \int_0^1
b_{xx}(\alpha, s, X_s+\theta(\widetilde{X}_s-X_s))d\theta Z_s^{h}Y_s
ds-\int_0^t b_{xx}(\alpha, s,X_s)Y^{2}_sds|^{2}]\\
&\leq &2\mathbb{\hat{E}}[(\int_0^T \int_0^1 |b_{xx}(\alpha, s,
X_s+\theta(\widetilde{X}_s-X_s))|d\theta |Z_s^{h}Y_s-Y^{2}_s| ds
)^{2}]\\&&+2\mathbb{\hat{E}}[(\int_0^T \int_0^1 |b_{xx}(\alpha, s,
X_s+\theta(\widetilde{X}_s-X_s))-b_{xx}(\alpha, s,X_s^{x})|d\theta
Y^{2}_sds)^{2}]\\
 &\leq & C\mathbb{\hat{E}}[(\int_0^T
|Z_s^{h}Y_s-Y^{2}_s| ds )^{2}]\\&&+C\mathbb{\hat{E}}[(\int_0^T
(|\widetilde{X}_s-X_s|+|\widetilde{X}_s-X_s|^{m+1}+
|X_s|^{m}|\widetilde{X}_s-X_s|) Y^{2}_sds)^{2}].
\end{eqnarray*}
For all $\varepsilon>0,$ from
$2ab\leq\frac{1}{\varepsilon}a^{2}+\varepsilon b^{2}, a, b\geq 0,$
it follows that
 \begin{eqnarray*}\label{}
&&\mathbb{\hat{E}}[\sup\limits_{t\in [0,T]}|\int_0^t \int_0^1
b_{xx}(\alpha, s, X_s+\theta(\widetilde{X}_s-X_s))d\theta Z_s^{h}Y_s
ds -\int_0^t b_{xx}(\alpha, s,X_s)Y^{2}_sds|^{2}]\\&\leq
&\frac{C}{\varepsilon}\mathbb{\hat{E}}[\sup\limits_{t\in [0,T]}
|Z_s^{h}-Y_s| ^{4}]+C\varepsilon\\&&+\frac{C}{\varepsilon}\int_0^T
\mathbb{\hat{E}}[|\widetilde{X}_s-X_s|^{4}+|\widetilde{X}_s-X_s|^{4m+4}+|X_s|^{4m}|\widetilde{X}_s-X_s|^{4}]ds.
\end{eqnarray*}
Using again Proposition \ref{pr3} and Proposition \ref{pr4}, we
obtain
\begin{eqnarray*}\label{}
&&\mathbb{\hat{E}}[\sup\limits_{t\in [0,T]}|\int_0^t \int_0^1
b_{xx}(\alpha, s, X_s+\theta(\widetilde{X}_s-X_s))d\theta Z_s^{h}Y_s
ds -\int_0^t b_{xx}(\alpha, s,X_s)Y^{2}_sds|^{2}]
\\&\leq&\frac{C}{\varepsilon}\mathbb{\hat{E}}[\sup\limits_{t\in [0,T]}
|Z_s^{h}-Y_s| ^{4}]+
 C\varepsilon+\frac{C}{\varepsilon} (h^{4}+h^{4m+4}).
\end{eqnarray*}
Step 4:
 Since \begin{eqnarray*}\label{}
\frac{1}{h}\int_0^t [b_{\alpha}(\alpha+h, s,
\widetilde{X}_s)-b_{\alpha}(\alpha, s,
\widetilde{X}_s)]ds=\int_0^t\int_0^1 b_{\alpha\alpha}(\alpha+\theta
h, s, \widetilde{X}_s)d\theta ds,
\end{eqnarray*}
then we have for some $n\in\mathbb{N}$,
\begin{eqnarray*}\label{}
&&\mathbb{\hat{E}}[\sup\limits_{t\in [0,T]}|\int_0^t\int_0^1
b_{\alpha\alpha}(\alpha+\theta h, s, \widetilde{X}_s)d\theta ds-\int_0^t b_{\alpha\alpha}(\alpha,s,X_s)ds|^{2}]\\
&\leq& C \mathbb{\hat{E}}[\int_0^T\int_0^1
|b_{\alpha\alpha}(\alpha+\theta h, s, X_s)-b_{\alpha\alpha}(\alpha,s,X_s)|^{2}d\theta ds]\\
&\leq&C(h^{2}+h^{2n+2}).
\end{eqnarray*}
Step 5: Since
 \begin{eqnarray*}\label{}
\frac{1}{h}\int_0^t [b_{\alpha}(\alpha, s,
\widetilde{X}_s)-b_{\alpha}(\alpha, s, X_s)]ds=\int_0^t\int_0^1
b_{x\alpha}(\alpha, s, (X_s+\theta(\widetilde{X}_s-X_s)))d\theta
Z_{s}^{h} ds,
\end{eqnarray*}
then we have for some $k\in\mathbb{N}$,
\begin{eqnarray*}\label{}
&&\mathbb{\hat{E}}[\sup\limits_{t\in [0,T]}|\int_0^t\int_0^1
b_{x\alpha}(\alpha, s,
(X_s+\theta(\widetilde{X}_s-X_s)))d\theta Z_{s}^{h} ds-\int_0^t b_{x\alpha}(\alpha,s,X_s)Y_sds|^{2}]\nonumber\\
&\leq& 2\mathbb{\hat{E}}[(\int_0^T\int_0^1 |b_{x\alpha}(\alpha, s,
(X_s+\theta(\widetilde{X}_s-X_s)))|d\theta |Z_{s}^{h}-Y_s|ds)^{2}]\\
&&+\mathbb{\hat{E}}[(\int_0^T\int_0^1|b_{x\alpha}(\alpha, s,
(X_s+\theta(\widetilde{X}_s-X_s)))-b_{x\alpha}(\alpha,s,X_s)|d\theta Y_sds)^{2}]\\
&\leq&C(h^{2}+h^{2k+2}).
\end{eqnarray*}
From Step 1-Step 5 it follows that
\begin{eqnarray*}\label{}
&&\mathbb{\hat{E}}[\sup\limits_{t\in [0,T]}|I_{1}-2\int_0^t
b_{x\alpha}(\alpha, s, X_s)Y_sds-\int_0^t b_{x}(\alpha, s,
X_s)P_sds\\&&-\int_0^t b_{\alpha\alpha}(\alpha,s,X_s)ds-\int_0^t b_{xx}(\alpha, s,X_s)Y^{2}_sds|^{2}]\\
&\leq &C\int_0^T \mathbb{\hat{E}}[\sup\limits_{r\in
[0,s]}|Q_r^{h}-P_r|^{2}] ds+C\varepsilon+\frac{C}{\varepsilon}
(h^{4}+h^{4l+4}+h^{4m+4}+h^{4k+4})\\&&+\frac{C}{\varepsilon}\mathbb{\hat{E}}[\sup\limits_{t\in
[0,T]} |Z_s^{h}-Y_s|
^{4}]+C(h^{2}+h^{2n+2}+h^{2k+2})+C\mathbb{\hat{E}}[\sup\limits_{s\in[0,T]}|\widetilde{Y}_s
-Y_s|^{2}].
\end{eqnarray*}
Using similar arguments we obtain that for some $k_{1},l_{1}, m_{1},
n_{1}\in\mathbb{N}$,
\begin{eqnarray*}\label{}
&&\mathbb{\hat{E}}[\sup\limits_{t\in [0,T]}|I_{2}-2\int_0^t
\sigma_{x\alpha}(\alpha, s, X_s)Y_sdB_{s}-\int_0^t
\sigma_{x}(\alpha, s, X_s)P_sdB_{s}\\&&-\int_0^t
\sigma_{\alpha\alpha}(\alpha,s,X_s)dB_{s}
-\int_0^t \sigma_{xx}(\alpha, s,X_s)Y^{2}_sdB_{s}|^{2}]\nonumber\\
&\leq &C\int_0^T \mathbb{\hat{E}}[\sup\limits_{r\in
[0,s]}|Q_r^{h}-P_r|^{2}] ds+C\varepsilon+\frac{C}{\varepsilon}
(h^{4}+h^{4l_{1}+4}+h^{4m_{1}+4}+h^{4k_{1}+4})\\&&+\frac{C}{\varepsilon}\mathbb{\hat{E}}[\sup\limits_{t\in
[0,T]} |Z_s^{h}-Y_s|
^{4}]+C(h^{2}+h^{2n_{1}+2}+h^{2k_{1}+2})+C\mathbb{\hat{E}}[\sup\limits_{s\in[0,T]}|\widetilde{Y}_s
-Y_s|^{2}].
\end{eqnarray*}
and for some $k_{2},l_{2}, m_{2}, n_{2}\in\mathbb{N}$,
\begin{eqnarray*}\label{}
&&\mathbb{\hat{E}}[\sup\limits_{t\in [0,T]}|I_{3}-2\int_0^t
h_{x\alpha}(\alpha, s, X_s)Y_sd\langle B\rangle_{s}-\int_0^t
h_{x}(\alpha, s, X_s)P_sd\langle B\rangle_{s}\\&&-\int_0^t
h_{\alpha\alpha}(\alpha,s,X_s)d\langle B\rangle_{s}-\int_0^t
h_{xx}(\alpha, s,X_s)Y^{2}_sd\langle
B\rangle_{s}|^{2}]\nonumber\\
&\leq &C\int_0^T \mathbb{\hat{E}}[\sup\limits_{r\in
[0,s]}|Q_r^{h}-P_r|^{2}] ds+C\varepsilon+\frac{C}{\varepsilon}
(h^{4}+h^{4l_{2}+4}+h^{4m_{2}+4}+h^{4k_{2}+4})\\&&+\frac{C}{\varepsilon}\mathbb{\hat{E}}[\sup\limits_{t\in
[0,T]} |Z_s^{h}-Y_s|
^{4}]+C(h^{2}+h^{2n_{2}+2}+h^{2k_{2}+2})+C\mathbb{\hat{E}}[\sup\limits_{s\in[0,T]}|\widetilde{Y}_s
-Y_s|^{2}].
\end{eqnarray*}
Then (\ref{e4}) and the above inequalities yield
\begin{eqnarray*}\label{}
&&\mathbb{\hat{E}}[\sup\limits_{t\in [0,T]}|Q_t^{h}-P_t|^{2}]\\
&\leq &C\int_0^T \mathbb{\hat{E}}[\sup\limits_{r\in
[0,s]}|Q_r^{h}-P_r|^{2}]ds+\frac{C}{\varepsilon}\mathbb{\hat{E}}[\sup\limits_{t\in
[0,T]} |Z_s^{h}-Y_s| ^{4}]+C\varepsilon\\
&&+\frac{C}{\varepsilon}(h^{4}+h^{4l+4}+h^{4m+4}+h^{4k+4}
+h^{4l_{1}+4}+h^{4m_{1}+4}+h^{4k_{1}+4}+h^{4l_{2}+4}+h^{4m_{2}+4}+h^{4k_{2}+4})\\
&&+C(h^{2}+h^{2n+2}+h^{2k+2}+h^{2n_{1}+2}+h^{2k_{1}+2}+h^{2n_{2}+2}+h^{2k_{2}+2})
+C\mathbb{\hat{E}}[\sup\limits_{s\in[0,T]}|\widetilde{Y}_s
-Y_s|^{2}].
\end{eqnarray*}
Thus, from Gronwall's inequality it follows that
\begin{eqnarray*}\label{}
&&\mathbb{\hat{E}}[\sup\limits_{t\in
[0,T]}|Q_t^{h}-P_t|^{2}]\\&\leq&
\frac{C}{\varepsilon}\mathbb{\hat{E}}[\sup\limits_{t\in
[0,T]} |Z_s^{h}-Y_s| ^{4}]+C\varepsilon+C(h^{2}+h^{2n+2}+h^{2k+2}+h^{2n_{1}+2}+h^{2k_{1}+2}+h^{2n_{2}+2}+h^{2k_{2}+2})\\
&&+\frac{C}{\varepsilon}(h^{4}+h^{4l+4}+h^{4m+4}+h^{4k+4}
+h^{4l_{1}+4}+h^{4m_{1}+4}+h^{4k_{1}+4}+h^{4l_{2}+4}+h^{4m_{2}+4}+h^{4k_{2}+4})
\\&&+C\mathbb{\hat{E}}[\sup\limits_{s\in[0,T]}|\widetilde{Y}_s
-Y_s|^{2}].
\end{eqnarray*}
Then by Remark \ref{re2} and Proposition \ref{pr5}, we get
\begin{eqnarray*}\label{}
\lim\limits_{h\rightarrow0}\mathbb{\hat{E}}[\sup\limits_{t\in
[0,T]}|Q_t^{h}-P_t|^{2}]\leq C\varepsilon.
\end{eqnarray*}
Therefore,
\begin{eqnarray*}\label{}
\lim\limits_{h\rightarrow0}\mathbb{\hat{E}}[\sup\limits_{t\in
[0,T]}|Q_t^{h}-P_t|^{2}]=0.
\end{eqnarray*}
  The proof is complete.
\end{proof}

  \vspace{4mm}

\noindent{\bf Acknowledgements.}

\vspace{4mm}This work is supported by Young Scholar Award for
Doctoral Students of the Ministry of Education of China and the
Marie Curie Initial Training Network (PITN-GA-2008-213841).

%    Insert the bibliography data here.

\end{document}